\documentclass[11pt]{amsart}

\usepackage{amsmath,amssymb,amsthm}
\usepackage{hyperref}

\newtheorem{theorem}{Theorem}[section]
\newtheorem{lemma}[theorem]{Lemma}

\title{Parametric Statistical Inference in the Zone of Moderate Deviation Probabilities}
\author{M.\,S. Ermakov\\[1ex]
\textit{Institute for Problems in Mechanical Engineering of the Russian Academy of Sciences\\
and St. Petersburg State University, St. Petersburg, Russia}
\texttt{erm2512@gmail.com}}
\date{October 27, 2026}
\thanks{The research was supported by the Ministry of Science and Higher Education of the Russian Federation (project 124041500008-1).}
\begin{document}

\begin{abstract}
A parametric theory of statistical inference is developed for the moderate deviation probability zone. The new approach to the proofs is based on a Taylor series expansion of the logarithm of the likelihood ratio based on the Hellinger distance. The Large Deviation Principle in the moderate deviation probability zone is proven for Bayesian estimators and maximum likelihood estimators. A uniform approximation of the logarithm of the likelihood ratio and a theorem on the concentration of the posterior Bayesian measure are also proven for the zone of moderate deviation probabilities.
\end{abstract}

\subjclass[2020]{62F12, 60F10, 62F15, 62E20}
\keywords{moderate deviation probabilities, maximum likelihood estimator, Bayes estimator, local asymptotic normality, large deviation principle, Hellinger distance}

\maketitle

\tableofcontents

\section{Introduction}

We constantly encounter small probabilities in statistics. Significance levels in confidence estimation and error probabilities in hypothesis testing take small values. The usual approach to studying such problems is to consider them as problems of large and moderate deviation probabilities. The question arises: under what conditions can the results on normal approximation of the distributions of statistics be extended to the zone of moderate deviation probabilities, at least for logarithmic asymptotics?

For maximum likelihood estimators and likelihood ratio tests, such results have been obtained in a number of works under rather stringent conditions \cite{arcones2002,borovkov1992,gao2001,ibragimov1981,miao2010,xiao2006} compared to those under which their asymptotic normality was established \cite{ibragimov1979,keener2010,lecam1986,vandervaart1998}. For Bayesian estimators, only concentration inequalities have been obtained \cite{ibragimov1979}. The logarithmic asymptotics corresponding to these inequalities \cite{ibragimov1979} are optimal in order.

Previously, to study the asymptotics of moderate deviation probabilities of maximum likelihood estimators and the likelihood ratio, a Taylor expansion of the logarithm of the likelihood ratio based on its derivatives was used \cite{arcones2002,borovkov1992,gao2001,ibragimov1981,miao2010,xiao2006}. In proving the local asymptotic normality of the likelihood ratio, along with this technique, a Taylor expansion of the logarithm of the likelihood ratio based on the Hellinger distance is also used \cite{hajek1972,ibragimov1979,lecam1986,vandervaart1998}. We apply this technique to study problems of statistical inference theory in the zone of moderate deviation probabilities. We have succeeded in extending to the zone of moderate deviation probabilities the main results obtained by this approach for normal approximation. We have proved:

\begin{itemize}
\item a theorem on local uniform normal approximation of the logarithm of the likelihood ratio in the zone of moderate deviation probabilities;
\item the large deviation principle in the zone of moderate deviation probabilities for maximum likelihood estimators and Bayesian estimators;
\item the large deviation principle in the zone of moderate deviation probabilities for the concentration of the posterior Bayesian measure.
\end{itemize}

The conditions of the theorems are much weaker compared to the conditions under which earlier theorems on moderate deviation probabilities of maximum likelihood estimators and the maximum of the likelihood ratio were proved \cite{arcones2002,borovkov1992,gao2001,ibragimov1981,miao2010,xiao2006}. For other formulations of the problems, such theorems, to the best of our knowledge, have not been obtained at all.

The above results are presented in Section 3. The main assumptions and the general formulation of the problem are given in Section 2. In the subsequent sections, their proofs are given sequentially.

We use the letters \(c\), \(C\) and the same letters \(c\), \(C\) with indices to denote positive constants. Denote by \(1(A)\) the indicator of the event \(A\). For any event \(A\), denote by \(\overline{A}\) the complementary event. For any two sequences of positive numbers \(a_n\) and \(b_n\), \(a_n = o(b_n)\) means \(a_n/b_n \to 0\) as \(n \to \infty\). For any vector \(u\), denote by \(u^T\) the transpose of the vector \(u\) and by \(|u|\) the length of the vector \(u\).

\section{Problem Statement. Main Assumptions}

Let \(X_1, \dots, X_n\) be a sample of independent identically distributed random variables with probability measure \(P_\theta\), \(\theta \in \Theta\), defined on the \(\sigma\)-algebra \(\mathcal{B}\) of the set \(S\). The set \(\Theta\) is an open bounded set in \(\mathbb{R}^d\). The value of the parameter \(\theta\) is unknown.

Assume that the probability measures \(P_\theta\), \(\theta \in \Theta\), are mutually absolutely continuous, have a common support which is the set \(S\), and are absolutely continuous with respect to a measure \(\nu\) defined on the same \(\sigma\)-algebra \(\mathcal{B}\) of the set \(S\). Denote
\begin{equation}
f(x,\theta) = \frac{dP_\theta}{d\nu}(x), \quad x \in S,
\end{equation}
the density of the probability measure \(P_\theta\), \(\theta \in \Theta\), with respect to the measure \(\nu\).

We are interested in the problem of statistical inference about the parameter \(\theta_0 \in \Theta\).

For all \(\theta_0\), \(\theta_0 + \tau \in \Theta\), define the function
\begin{equation}
g(x,\tau) = g(x,\theta_0,\theta_0 + \tau) = \left( \frac{f(x,\theta_0 + \tau)}{f(x,\theta_0)} \right)^{1/2} - 1,
\end{equation}
where \(x \in S\).

We say that the statistical experiment \(\mathcal{E} = \{(S,\mathcal{B}), P_\theta, \theta \in \Theta\}\) has finite Fisher information at the point \(\theta_0 \in \Theta\) if there exists a vector function \(\varphi = \varphi_{\theta_0}: S \to \mathbb{R}^d\) and an increasing function \(\omega(|\tau|) = \omega_{\theta_0}(|\tau|) \to 0\) as \(\tau \to 0\) such that
\begin{equation}
\int_S \bigl( g(x,\tau) - \tfrac12 \tau^T \varphi(x) \bigr)^2 \, dP_{\theta_0} < |\tau|^2 \omega(|\tau|).
\end{equation}

Define the Fisher information matrix
\begin{equation}
I(\theta_0) = \int_S \varphi \varphi^T \, dP_{\theta_0}.
\end{equation}

We will assume that the rank of the matrix \(I(\theta)\) is equal to \(d\) and the elements of the matrix \(I(\theta)\) are continuous in \(\theta\) with modulus of continuity \(\omega_\theta(|u|)\) at all points \(\theta \in \Theta \subset \mathbb{R}^d\).

We introduce the following condition, which is essentially a condition of identifiability of the parameter value (see condition (3.2), Ch. 3, \cite{ibragimov1979}).

\textbf{A0.} For any compact \(K \subset \Theta\) and any \(\delta > 0\)
\begin{equation}
\inf_{\theta \in K} \inf_{\substack{\tau:\ \theta+\tau \in \Theta \\ |\tau| \ge \delta}} \int_S \bigl( f^{1/2}(x,\theta+\tau) - f^{1/2}(x,\theta) \bigr)^2 \, d\nu > 0
\end{equation}
and
\begin{equation}
\sup_{\theta \in \Theta} \sup_{\substack{\tau:\ |\tau|>0 \\ \theta+\tau \in \Theta}} |\tau|^{-2} \int_S \bigl( f^{1/2}(x,\theta+\tau) - f^{1/2}(x,\theta) \bigr)^2 \, d\nu < \infty.
\end{equation}

Condition (5) is essentially a condition of identifiability of the parameter value (see condition (3.2), Ch. 3, \cite{ibragimov1979}). Condition (6) will be necessary so that we can apply the concentration inequalities for maximum likelihood estimators and Bayesian estimators proved in Theorems 5.1--5.7 in \cite{ibragimov1979}.

Let a sequence \(u_n > 0\), \(u_n \to 0\), \(n u_n^2 \to \infty\) as \(n \to \infty\) be given. We are interested in the moderate deviation probabilities \(P_{\theta_0}(|\hat\theta_n - \theta_0| > u_n)\), where \(\hat\theta_n\) is a maximum likelihood estimator or a Bayesian estimator. We will also study the moderate deviation probabilities of the concentration of the posterior Bayesian measure in the region \(|\theta - \theta_0| > u_n\).

The following conditions \(A1\) and \(A2\) (\(B1\) and \(B2\) respectively) give the statements of the theorems depending on the rate of convergence of the function \(\omega(u)\) to zero as \(u \to 0\). Weaker moment conditions \(A1\) (\(B1\) respectively) will be assumed when \(\omega(u) < C u^\lambda\) for some \(\lambda > 0\).

\textbf{A1.} For any \(\theta_0 \in \Theta\) and any \(C, C_1\), there exist positive constants \(C_2, \gamma, n_0(C,C_1)\), a sequence \(\kappa_n = \kappa_n(u_n)\), \(\kappa_n \to 0\) as \(n \to \infty\), and a sequence of functions \(h_n: S \to \mathbb{R}_+\) such that for \(n > n_0(C,C_1)\), almost everywhere with respect to the measure \(P_{\theta_0}\),
\begin{equation}
\sup_{|\theta - \theta_0| < C u_n} \sup_{|\tau| < C_1 u_n} \log \frac{f(x,\theta+\tau)}{f(x,\theta)} \le \kappa_n h_n(x), \quad x \in S,
\end{equation}
and \(\mathbb{E}_\theta \exp\{\gamma h_n(X_1)\} < C_2\) for all \(\theta\), \(|\theta - \theta_0| < C u_n\).

\textbf{A2.} For any \(\theta_0 \in \Theta\) and any \(C, C_1\), there exist positive constants \(n_0(C,C_1), C_2, c, \gamma\), a sequence \(\kappa_n = \kappa_n(u_n)\), \(\kappa_n |\log u_n| < c\), and a sequence of functions \(h_n: S \to \mathbb{R}_+\) such that (7) holds and \(\mathbb{E}_\theta \exp\{\gamma h_n(X_1)\} < C_2\) for all \(\theta\), \(|\theta - \theta_0| < C u_n\).

Conditions \(A1\) and \(A2\) for \(\kappa_n = \gamma / \gamma_n\) are sufficient for the more difficult to verify conditions \(B1\) and \(B2\) respectively, under which the statements of the theorems will be proved.

\textbf{B1.} For any \(\theta_0 \in \Theta\) and any \(C, C_1\), there exist constants \(n_0(C,C_1), \varepsilon > 0, C_2\) and a sequence \(\gamma_n = \gamma_n(u_n)\), \(\gamma_n \to \infty\) as \(n \to \infty\), such that for \(n > n_0(C,C_1)\),
\begin{equation}
\sup_{|\theta - \theta_0| < C u_n} \sup_{|\tau| < C_1 u_n} \mathbb{E}_\theta \left( \frac{f(X,\theta+\tau)}{f(X,\theta)} \, 1\bigl( \bigl|\log \tfrac{f(X,\theta+\tau)}{f(X,\theta)}\bigr| > \varepsilon \bigr) \right)^{\gamma_n} < C_2.
\end{equation}

\textbf{B2.} For any \(\theta_0 \in \Theta\) and any \(C, C_1\), there exist constants \(n_0(C,C_1), \varepsilon > 0, C_2\) and a sequence \(\gamma_n = \gamma_n(u_n)\), \(\gamma_n > c |\log u_n|\), such that (8) holds.

In the proofs of the theorems, conditions \(A1\) and \(A2\) are used when \(\bigl|\log \frac{f(X,\theta+u)}{f(X,\theta)}\bigr| > \varepsilon\). This means that in fact it is sufficient to require their fulfillment in the regions \(h_n(x) > \varepsilon / \kappa_n\).

Condition (8) is derived from condition (7) quite simply. We have
\begin{equation}
\mathbb{E}_\theta \left( \frac{f(X_1,\theta+\tau)}{f(X_1,\theta)} \right)^{\gamma_n} \le \mathbb{E}_\theta \exp\{\kappa_n \gamma_n h_n(X_1)\} = \mathbb{E}_\theta \exp\{\gamma h_n(X_1)\},
\end{equation}
if we take \(\kappa_n = \gamma / \gamma_n\).

Conditions \(B1\), \(B2\) do not require the existence of an exponential moment of \(\sup |\partial_{\theta\theta} \log f|\). Consider the shift family \(f(x,\theta) = c_\alpha \exp(-|x-\theta|^\alpha)\), \(\alpha \in (1/2,1]\). The Fisher information is finite for \(\alpha > 1/2\). For \(0 < \alpha \le 1\), it is true that \(\bigl||y+h|^\alpha - |y|^\alpha\bigr| \le |h|^\alpha\), therefore
\[
\bigl|\log f(x,\theta+\tau) - \log f(x,\theta)\bigr| \le |\tau|^\alpha
\]
uniformly in \(x\). Condition \(A1\) is satisfied with \(\kappa_n = u_n^\alpha\), \(h_n \equiv 1\); for large \(n\), the indicator in (7) is identically zero. The second derivative of \(\log f\) either does not exist as a function (\(\alpha=1\)) or has a non-integrable singularity \(|x-\theta|^{\alpha-2}\) for \(\alpha < 1\), so that \(\mathbb{E} \exp\{\lambda \sup |\partial_{\theta\theta} \log f|\} = \infty\).

For the proof of the theorems when condition \(A1\) or \(B1\) holds, we will additionally require the fulfillment of conditions \(C0\)--\(C2\).

\textbf{C0.} For any \(\theta_0 \in \Theta\), for some neighborhood \(U \subset \Theta\) of the point \(\theta_0\), there exist \(\delta > 0\) and \(\gamma > 0\) such that for any \(s\), \(-\delta < s < \delta\), and all \(\theta \in U\),
\begin{equation}
\int_S f^s(x,\theta) f^{1-s}(x,\theta+u) \, 1(|g(x,u)| > \varepsilon) \, d\nu < C \varepsilon |u|^{2+\gamma}
\end{equation}
for \(u\) from some neighborhood of zero.

\textbf{C1.} For any \(\theta_0 \in \Theta\), there exist a constant \(\gamma > 0\) and \(C\) such that for all \(t\), \(0 < t < t_0\), \(\omega(t) < C t^\gamma\), and for all \(\theta\) from some neighborhood \(U\) of the point \(\theta_0\),
\begin{equation}
\mathbb{E}_\theta [|\varphi(X_1)|^{2+\gamma}] < C.
\end{equation}

\textbf{C2.} For any \(\theta_0 \in \Theta\), there exist \(\gamma > 0\) and \(C > 0\) such that for all \(\tau \in \mathbb{R}^d\), \(0 < |\tau| < u_0\), \(\theta_0 + \tau \in \Theta\),
\begin{equation}
\bigl| \mathbb{E}_{\theta_0} [g^2(X_1,\tau)] - \tfrac14 \tau^T I(\theta_0) \tau \bigr| < C |\tau|^{2+\gamma}.
\end{equation}

Until now, for the proof of the theorem on moderate deviation probabilities of maximum likelihood estimators and maximum likelihood ratio tests, a Taylor expansion of the logarithm of the likelihood ratio function
\begin{equation}
\log \frac{f(X,\theta+u)}{f(X,\theta)} = u \frac{\partial \log f(X,\theta+u)}{\partial u}\Big|_{u=0} + \frac{u^2}{2} \frac{\partial^2 \log f(X,\theta+u)}{\partial^2 u}\Big|_{u=v}
\end{equation}
for \(v \in (0,u)\) was used. Assumptions were made about the existence of exponential moments of the suprema over the parameter values of the derivatives of the logarithm of the likelihood ratio. Such conditions were considered for derivatives up to the third order \cite{arcones2002,borovkov1992,gao2001,ibragimov1981,miao2010,xiao2006}.

The weakest assumptions appear to be contained in the work \cite{arcones2002}. In \cite{arcones2002}, the theorem on the moderate deviation probability of the maximum likelihood estimator is proved under the assumption of the existence of an exponential moment of the supremum over the parameter values of the second derivative of the logarithm of the likelihood ratio. The same statement was proved in \cite{arcones2002} when only the existence of an exponential moment of the supremum of the first derivative was required, but it was assumed that the likelihood ratio function of the sample with probability one has one local maximum coinciding with the global one. In the formulation of the problem of this work, the condition of the existence of an exponential moment of the supremum of the first derivative can be considered as a special case of condition \(A1\) for \(\kappa_n = u_n\) under some additional not very restrictive conditions. Moreover, we do not introduce any assumptions about the uniqueness of the local maximum of the likelihood ratio function coinciding with the global one.

The following condition is used in the proof of theorems on moderate deviation probabilities of maximum likelihood estimators in the multidimensional case. For Bayesian estimators, it is absent.

\textbf{D.} The set \(\Theta\) is convex, bounded, and there exists \(m > d\) such that
\begin{equation}
\sup_{\theta \in \Theta} \int_S \left( \frac{|\nabla f(x,\theta)|}{f(x,\theta)} \right)^m f(x,\theta) \, d\nu < \infty.
\end{equation}
Here \(\nabla f(x,\theta)\) denotes the gradient of the function \(f(x,\theta)\) with respect to the parameter \(\theta\).

For \(\varepsilon > 0\), set \(\varphi_\varepsilon(x) = \varphi(x) 1(u_n |\varphi(x)| < \varepsilon)\), \(x \in S\).

The following condition is used in the proof of the uniform convergence of the logarithm of the likelihood ratio to its locally asymptotically normal approximation in the zone of moderate deviation probabilities.

\textbf{E.} For any \(\theta_0 \in \Theta\), there exist \(\beta_1 > d\) and \(\beta_2 > d\), \(\beta_2 \le 2\beta_1\), such that for some neighborhood \(U\) of the point \(\theta_0\), for any \(\theta\), \(\theta_0 + u\), \(\theta_0 + v \in U\),
\begin{equation}
\mathbb{E}_\theta \Bigl| \log \frac{f(X_1,\theta_0 + v)}{f(X_1,\theta_0 + u)} - (v-u)^T \varphi(X_1) 1(u |\varphi(x)| < \varepsilon) \Bigr|^{\beta_1} \le C |v-u|^{\beta_2}.
\end{equation}

\section{Main Results}

\subsection{Local asymptotic normality of the logarithm of the likelihood ratio in the zone of moderate deviation probabilities}

For \(u \in \mathbb{R}^d\) and \(b \in \mathbb{R}^d\), denote
\begin{equation}
\zeta_n(u_n b, u) = u^T \sum_{i=1}^n \varphi_\varepsilon(\theta_0 + u_n b, X_i) - \tfrac12 n u^T I(\theta_0 + u_n b) u
\end{equation}
and
\begin{equation}
\psi_n = n^{-1/2} I^{-1/2}(\theta_0) \sum_{i=1}^n \varphi_\varepsilon(X_i),
\end{equation}
where \(\varphi_\varepsilon(\theta_0 + u_n b, x) = \varphi_{\theta_0 + u_n b}(x) 1(|\varphi_{\theta_0 + u_n b}(x)| < u_n^{-1} \varepsilon)\), \(x \in S\). In the case of condition \(B1\), the choice of \(\varepsilon > 0\) is arbitrary. In the case of condition \(B2\), \(\varepsilon > 4/c\) is taken, where \(c\) is from condition \(B2\).

We will abbreviate the notation and write \(\varphi_\varepsilon(x) = \varphi_\varepsilon(\theta_0, x)\) and \(\zeta_n(u) = \zeta_n(0, u)\).

For \(\theta_0\), \(\theta_0 + u_n b\), \(\theta_0 + u_n b + u \in \Theta\), \(b \in \mathbb{R}^d\), set
\begin{equation}
\xi_n(X_i, u_n b, u) = \log \frac{f(X_i, \theta_0 + u_n b + u)}{f(X_i, \theta_0 + u_n b)}, \quad 1 \le i \le n.
\end{equation}
If \(b = 0\), we will denote \(\xi_n(X_i, u) = \xi_n(X_i, 0, u)\).

The following Theorem 3.1 is an analogue of the theorem on local asymptotic normality of the logarithm of the likelihood ratio \cite{hajek1972,ibragimov1979,lecam1986,vandervaart1998} for the zone of moderate deviation probabilities.

\begin{theorem}
\label{thm:3.1}
Let the statistical experiment have finite Fisher information at the point \(\theta_0 \in \Theta \subset \mathbb{R}^d\) and conditions \(A0\), \(B2\) be satisfied. Then for any \(C\), there exist positive functions \(\omega_1(u)\), \(\omega_1(u) \to 0\) as \(u \to 0\), and \(\omega_2(u)\), \(\omega_2(u) \to 0\) as \(u \to 0\), such that for \(|u| < C u_n\) and \(b \in \mathbb{R}^d\), for the vector \(\theta_n = \theta_0 + u_n b\),
\begin{align}
&P_{\theta_n} \Bigl( \Bigl| \sum_{i=1}^n \xi_n(X_i, u_n b, u) - \zeta_n(u) + n u_n u^T I(\theta_0) b \Bigr| > n u_n^2 \omega_1(u_n) \Bigr) \nonumber \\
&\qquad < C_1 \exp\bigl\{ -n u_n^2 / \omega_2(u_n) \bigr\}
\end{align}
and
\begin{align}
&P_{\theta_n} \Bigl( \Bigl| \sum_{i=1}^n \xi_n(X_i, u_n b, u) - \zeta_n(u_n b, u) \Bigr| > n u_n^2 \omega_1(u_n) \Bigr) \nonumber \\
&\qquad < C_1 \exp\bigl\{ -n u_n^2 / \omega_2(u_n) \bigr\}.
\end{align}
Suppose additionally that condition \(E\) is satisfied. Then for any \(C\), there exist positive functions \(\omega_3(u)\), \(\omega_3(u) \to 0\) as \(u \to 0\), and \(\omega_4(u)\), \(\omega_4(u) \to 0\) as \(u \to 0\), such that
\begin{align}
&P_{\theta_n} \Bigl( \sup_{|u| < C u_n} \Bigl| \sum_{i=1}^n \xi_n(X_i, u_n b, u) - \zeta_n(u) + n u_n u^T I(\theta_0) b \Bigr| > n u_n^2 \omega_3(u_n) \Bigr) \nonumber \\
&\qquad < C_1 \exp\bigl\{ -n u_n^2 / \omega_4(u_n) \bigr\}
\end{align}
and
\begin{align}
&P_{\theta_n} \Bigl( \sup_{|u| < C u_n} \Bigl| \sum_{i=1}^n \xi_n(X_i, u_n b, u) - \zeta_n(u_n b, u) \Bigr| > n u_n^2 \omega_3(u_n) \Bigr) \nonumber \\
&\qquad < C_1 \exp\bigl\{ -n u_n^2 / \omega_4(u_n) \bigr\}.
\end{align}
All the above inequalities remain valid if condition \(B2\) is replaced by conditions \(B1\) and \(C0\)--\(C2\). Estimates (16)--(19) hold uniformly in \(b\), \(|b| < C\) for any \(C > 0\).
\end{theorem}

Condition \(E\) is needed to estimate the modulus of continuity of the logarithm of the likelihood ratio.

The arguments will be based on the following theorem (see Theorem 2.1, \cite{ermakov2003}).

\begin{theorem}
\label{thm:3.2}
Let the statistical experiment have finite Fisher information at the point \(\theta_0 \in \Theta \subset \mathbb{R}^d\). Then for any open set \(\Omega \subset \mathbb{R}^d\),
\begin{equation}
\lim_{n\to\infty} (n u_n^2)^{-1} \log P_{\theta_0}(\psi_n \in n^{1/2} u_n \Omega) = -\tfrac12 \inf_{u \in \Omega} |u|^2
\end{equation}
and for any vector \(b \in \mathbb{R}^d\),
\begin{equation}
\lim_{n\to\infty} (n u_n^2)^{-1} \log P_{\theta_0 + u_n b}(\psi_n \in n^{1/2} u_n I^{1/2}(\theta_0) b + n^{1/2} u_n \Omega) = -\tfrac12 \inf_{u \in \Omega} |u|^2.
\end{equation}
\end{theorem}

The proof of Theorem 3.2 is a simple technical application of Theorem II.2 of Ellis \cite{ellis1984} with the normalization \(a_n = n u_n^2\) therein. It suffices only to note that
\begin{equation}
\lim_{n\to\infty} (n u_n^2)^{-1} \log \mathbb{E}_{\theta_0} \exp\bigl\{ u_n t^T \sum_{i=1}^n \varphi_\varepsilon(X_i) \bigr\} = \tfrac12 t^T I(\theta_0) t, \quad t \in \mathbb{R}^d.
\end{equation}

\subsection{Maximum likelihood estimator}

In what follows, we will assume that for all \(x \in S\), the density \(f(x,\theta)\) is continuous in \(\theta\) on the set \(\Theta\).

Define the maximum likelihood estimator \(\hat\theta_n = \hat\theta_n(X_1,\dots,X_n)\) by the equation
\begin{equation}
\prod_{i=1}^n f(X_i, \hat\theta_n) = \max_{\theta \in \Theta} \prod_{i=1}^n f(X_i, \theta),
\end{equation}
where the maximum likelihood estimator \(\hat\theta_n\) is taken to be any of the points at which the maximum is attained, so that the estimator remains measurable. In the proof of Theorem 3.3, it is shown that under its conditions, all such maximum points lie in a ball of radius \(o(u_n)\) with probability \(1 - \exp\{-C_n n u_n^2\}\), where \(C_n \to \infty\) as \(n \to \infty\). Therefore, we will not discuss the question of the measurable choice of one of such points and the construction of the estimator when the maximum is attained on the boundary of the domain.

\begin{theorem}
\label{thm:3.3}
Let the statistical experiment have finite Fisher information at the point \(\theta_0 \in \Theta\) and conditions \(A0\) and \(B2\) be satisfied. Then, if \(d=1\) and the set \(\Theta\) is a bounded interval in \(\mathbb{R}^1\), for any open set \(\Omega \subset \mathbb{R}^1\),
\begin{equation}
\lim_{n\to\infty} (n u_n^2 / 2)^{-1} \log P_{\theta_0}\bigl( I^{1/2}(\theta_0)(\hat\theta_n - \theta_0) \in u_n \Omega \bigr) = -\inf_{x \in \Omega} |x|^2
\end{equation}
and
\begin{equation}
\lim_{n\to\infty} (n u_n^2 / 2)^{-1} \log P_{\theta_0 + u_n}\bigl( I^{1/2}(\theta_0)(\hat\theta_n - \theta_0 - u_n) \in u_n \Omega \bigr) = -\inf_{x \in \Omega} |x|^2.
\end{equation}
If \(d > 1\) and condition \(D\) is additionally satisfied, then for any open set \(\Omega \subset \mathbb{R}^d\), (24) holds, and for any vector \(b \in \mathbb{R}^d\),
\begin{equation}
\lim_{n\to\infty} (n u_n^2 / 2)^{-1} \log P_{\theta_0 + u_n b}\bigl( I^{1/2}(\theta_0)(\hat\theta_n - \theta_0 - u_n b) \in u_n \Omega \bigr) = -\inf_{x \in \Omega} |x|^2.
\end{equation}
For any \(\delta > 0\) and any vector \(b \in \mathbb{R}^d\),
\begin{equation}
\lim_{n\to\infty} (n u_n^2 / 2)^{-1} \log P_{\theta_0 + u_n b}\bigl( |I^{1/2}(\theta_0)(\hat\theta_n - \theta_0) - n^{-1/2} \psi_n| > \delta u_n \bigr) = -\infty,
\end{equation}
\begin{align}
&\lim_{n\to\infty} (n u_n^2 / 2)^{-1} \log P_{\theta_0 + u_n b}\Bigl( \Bigl| \sum_{i=1}^n \xi_n(X_i, u_n b, \hat\theta_n - \theta_0 - u_n b) \nonumber \\
&\qquad - \tfrac12 (\psi_n - n^{1/2} u_n I^{1/2}(\theta_0) b)^T (\psi_n - n^{1/2} u_n I^{1/2}(\theta_0) b) \Bigr| > \delta n u_n^2 \Bigr) = -\infty
\end{align}
and
\begin{align}
&\lim_{n\to\infty} (n u_n^2 / 2)^{-1} \log P_{\theta_0 + u_n b}\Bigl( \Bigl| \sum_{i=1}^n \xi_n(X_i, u_n b, \hat\theta_n - \theta_0 - u_n b) \nonumber \\
&\qquad - n (\hat\theta_n - \theta_0 - u_n b)^T I(\theta_0) (\hat\theta_n - \theta_0 - u_n b)/2 \Bigr| > \delta n u_n^2 \Bigr) = -\infty.
\end{align}
All the above statements remain valid if condition \(B2\) is replaced by conditions \(B1\) and \(C0\)--\(C2\).
\end{theorem}

For the formulation of the problem of moderate deviation probabilities of the estimator, (27)--(29) show the equivalence of the application of some variants of the maximum likelihood ratio, Wald, and Rao tests for multidimensional hypothesis testing problems about the value of the parameter. In this case, truncated functions \(\varphi\) are considered. To prove an analogous statement for the tests themselves, one needs to estimate the contribution of the tails of the function \(\varphi\) to the sum. In particular, such a replacement by the function \(\varphi\) is legitimate if for all \(\theta\) from some neighborhood of \(\theta_0\), \(\mathbb{E}_\theta \exp\{c |\varphi(X_1)|\} < C\) for some \(c > 0\).

\begin{theorem}
\label{thm:3.4}
Let the statistical experiment have finite Fisher information at the point \(\theta_0 \in \Theta\) and conditions \(A0\), \(B2\) or \(A0\), \(B1\), \(C0\)--\(C2\) be satisfied, where the corresponding conditions \(B1\) or \(B2\) must be satisfied for any sequence \(u_n \to 0\), \(n u_n^2 \to \infty\) as \(n \to \infty\).

Then, if \(d=1\) and the set \(\Theta\) is a bounded interval in \(\mathbb{R}^1\), for any open set \(\Omega \subset \mathbb{R}^1\),
\begin{equation}
\lim_{\substack{n\to\infty \\ C\to\infty \\ c\to 0}} \sup_{\substack{n |u|^2 > C \\ |u| < c}} \bigl| (n |u|^2 / 2)^{-1} \log P_{\theta_0}\bigl( I^{1/2}(\theta_0)(\hat\theta_n - \theta_0) \in u \Omega \bigr) + \inf_{x \in \Omega} |x|^2 \bigr| = 0.
\end{equation}
If \(d > 1\) and condition \(D\) is additionally satisfied, then for any open set \(\Omega \subset \mathbb{R}^d\), (30) holds, and for any \(\delta > 0\),
\begin{align}
&\limsup_{\substack{n\to\infty \\ C\to\infty \\ c\to 0}} \sup_{\substack{n |u|^2 > C \\ |u| < c}} (n |u|^2)^{-1} \log P_{\theta_0}\bigl( |I^{1/2}(\theta_0)(\hat\theta_n - \theta_0) - n^{-1/2} \psi_n| > \delta |u| \bigr) = -\infty, \\
&\limsup_{\substack{n\to\infty \\ C\to\infty \\ c\to 0}} \sup_{\substack{n |u|^2 > C \\ |u| < c}} (n |u|^2)^{-1} \log P_{\theta_0}\Bigl( \Bigl| \sum_{i=1}^n \xi_n(X_i, \hat\theta_n - \theta_0) - \tfrac12 \psi_n^T \psi_n \Bigr| > \delta n |u|^2 \Bigr) = -\infty
\end{align}
and
\begin{align}
&\limsup_{\substack{n\to\infty \\ C\to\infty \\ c\to 0}} \sup_{\substack{n |u|^2 > C \\ |u| < c}} (n |u|^2)^{-1} \log P_{\theta_0}\Bigl( \Bigl| \sum_{i=1}^n \xi_n(X_i, 0, \hat\theta_n - \theta_0) \nonumber \\
&\qquad - n (\hat\theta_n - \theta_0)^T I(\theta_0) (\hat\theta_n - \theta_0)/2 \Bigr| > \delta n |u|^2 \Bigr) = -\infty.
\end{align}
Here \(\varphi_\varepsilon(x) = \varphi(x) 1(|\varphi(x)| < \varepsilon |u|^{-1})\). In (30)--(33), the first limits are taken as \(n\), \(C\) simultaneously tend to infinity and \(c\) to zero.
\end{theorem}

Note that for the formulation of the problem of Theorem 3.4 under conditions \(B1\) and \(B2\), one can completely remove the dependence on \(n\) by omitting this index, assuming that the function \(\gamma = \gamma(u)\) is such that \(\gamma(u) \to \infty\) as \(u \to 0\) and (8) holds for all \(|u| < u_0\), \(u_0 > 0\). In the analogue of condition \(B2\), one must additionally require \(\gamma(u) > c |\log u|\).

Note that from (24) the following weak analogues of the statement on the local asymptotic Bahadur efficiency of the maximum likelihood estimator follow:
\begin{equation}
\liminf_{u\to 0} \liminf_{n\to\infty} (n u^2 / 2)^{-1} \log P_{\theta_0}\bigl( I^{1/2}(\theta_0)(\hat\theta_n - \theta_0) \in u \Omega \bigr) = -\inf_{x \in \Omega} |x|^2
\end{equation}
and
\begin{equation}
\limsup_{u\to 0} \limsup_{n\to\infty} (n u^2 / 2)^{-1} \log P_{\theta_0}\bigl( I^{1/2}(\theta_0)(\hat\theta_n - \theta_0) \in u \Omega \bigr) = -\inf_{x \in \Omega} |x|^2.
\end{equation}

In Theorems 3.3 and 3.4, the set \(\Theta\) is assumed to be bounded. In Theorems 5.1 and 5.4 in Ch. 1 of \cite{ibragimov1979}, inequalities for the moderate deviation probabilities of maximum likelihood estimators are obtained when the parameter set is unbounded. These inequalities allow one to extend the proofs of Theorems 3.3 and 3.4 to the case of an unbounded set \(\Theta\), since they give the behavior of the estimator outside this set. Thus (24)--(35) are valid for unbounded domains \(\Theta\) if the conditions (1)--(3) of Theorem 5.4 in Ch. 1 of \cite{ibragimov1979} are additionally satisfied. It is easy to see that they can be replaced by the following simpler sufficient conditions.

There exists \(\gamma > 0\) such that for any compact \(K \subset \Theta\), there exists \(c = c(K)\) such that
\begin{equation}
\sup_{|h| > R} \int_S f^{1/2}(x,\theta) f^{1/2}(x,\theta + h) \, d\nu \le c R^{-\gamma}, \quad \theta \in K.
\end{equation}
Moreover, for some \(m > 0\), the following conditions must be satisfied:
\begin{align}
&\sup_{|\theta| < R} \mathbb{E}_\theta |\varphi_\theta(X_1)|^2 < C(1 + R^m), \quad \inf_{\theta \in \Theta} \mathbb{E}_\theta \varphi_\theta^2(X_1) > c > 0, \nonumber \\
&\sup_{|\theta| < R} \sup_h \omega_\theta(h) < C(1 + R^m).
\end{align}
Theorem 5.4, Ch. 1 of \cite{ibragimov1979} is formulated and proved for \(\Theta \subset \mathbb{R}^1\). However, it is also valid for \(\Theta \subset \mathbb{R}^d\) if condition \(D\) is additionally assumed. In the multidimensional case, the arguments of the proof differ only in the estimate of the modulus of continuity of the likelihood ratio, which is provided by condition \(D\).

\subsection{Bayesian estimator}

Let \(l_1: \mathbb{R}_+ \to \mathbb{R}_+\) be an increasing function such that \(l_1(0) = 0\) and \(l_1(x) > 0\) if \(x \ne 0\). Let the function \(l_1\) satisfy the following condition. There exist positive constants \(C_1\), \(C_2\) and \(\kappa_1\), \(\kappa_2\) such that for any \(a > 1\) and \(x > 0\),
\begin{equation}
l_1(a x) \le C_1 a^{\kappa_1} l_1(x), \qquad C_2 (a-1)^{\kappa_2} < \frac{l_1(a x) - l_1(x)}{l_1(x)}.
\end{equation}
The loss function \(l(u)\) has the form \(l(u) = l_1(\|u\|)\), where \(\|u\|\) is some (not necessarily Euclidean) norm of the vector \(u \in \mathbb{R}^d\).

We do not normalize the deviation \(\tilde\theta_n - \theta_0\) of the Bayesian estimator by \(\sqrt{n}\) inside the loss function in the definition of the Bayesian risk, as is done when studying its asymptotic normality \cite{ibragimov1979,vandervaart1998}. As a result, in fact, one only needs to require the existence for \(x\) of some neighborhood of zero such that the second inequality in (38) holds for \(a < 100\). In order not to increase the volume of the work, for estimating large deviation probabilities we use the concentration inequalities for Bayesian estimators from \cite{ibragimov1979}, and in \cite{ibragimov1979} it is assumed that the loss function grows at most polynomially.

For convenience of estimation, we will assume that \(1 = \inf\{ |I^{1/2}(\theta_0) u| : \|u\| = 1, u \in \mathbb{R}^d \}\).

We will assume that the Bayesian prior density \(\pi(\theta)\), \(\theta \in \Theta\), is continuous, bounded, and \(\pi(\theta) > 0\) for all \(\theta \in \Theta\).

The Bayesian estimator of the parameter \(\theta\) is given by the relation
\begin{equation}
\tilde\theta_n = \arg\min_{t \in \Theta} \int_\Theta l(t - \theta) q_n(\theta) \, d\theta,
\end{equation}
where
\begin{equation}
q_n(\theta) = \frac{ \prod_{j=1}^n f(X_j,\theta) \pi(\theta) }{ \int_\Theta \prod_{j=1}^n f(X_j,\theta) \pi(\theta) \, d\theta }
\end{equation}
is the posterior density.

\begin{theorem}
\label{thm:3.5}
Let the statistical experiment have finite Fisher information at the point \(\theta_0 \in \Theta \subset \mathbb{R}^d\) and conditions \(A0\), \(B2\) or \(A0\), \(B1\), \(C0\)--\(C2\) be satisfied. Then for the Bayesian estimator \(\tilde\theta_n\), the same statements (24)--(33) of Theorems 3.3 and 3.4 hold. It suffices only to replace in them the maximum likelihood estimator \(\hat\theta_n\) by the Bayesian estimator \(\tilde\theta_n\). Conditions \(D\) or \(E\) for \(d > 1\) are assumed here only in the analogues of statements (28), (29) and (32), (33).
\end{theorem}

The inequalities of Theorems 5.2 and 5.6 in Ch. 1 of \cite{ibragimov1979} allow one to extend Theorem 3.5 to the case when the set \(\Theta \subset \mathbb{R}^d\) is unbounded. Theorem 3.5 for an unbounded set \(\Theta \subset \mathbb{R}^d\) is valid if the conditions (1)--(3) of Theorem 5.6 in Ch. 1 of \cite{ibragimov1979} are additionally satisfied. Sufficient conditions for their fulfillment are the conditions (36) and (37). It is assumed that the prior density \(\pi(\theta)\) grows at infinity no faster than polynomially.

Theorem 5.6 in Ch. 1 of \cite{ibragimov1979} is proved for the case \(l(u) = l_1(|u|)\). However, the proof extends practically without changes to the case \(l(u) = l_1(\|u\|)\). We omit the corresponding arguments.

\subsection{On moderate deviation probabilities of the concentration of the posterior Bayesian measure}

For any measurable set \(A \subset \Theta\), the posterior Bayesian measure of the set \(A\) is equal to
\begin{equation}
Q_n(A) = \int_A q_n(\theta) \, d\theta.
\end{equation}

\begin{theorem}
\label{thm:3.6}
Let the statistical experiment have finite Fisher information at the point \(\theta_0 \in \Theta \subset \mathbb{R}^d\) and conditions \(A0\), \(B2\) or conditions \(A0\), \(B1\) and \(C0\)--\(C2\) be satisfied. Then for any open set \(\Omega \subset \mathbb{R}^d\), \(\inf\{ |x| : x \in \Omega \} > 0\), for any constant \(c\), \(0 < c < 1\), and any vector \(b \in \mathbb{R}^d\),
\begin{equation}
\lim_{n\to\infty} (n u_n^2 / 2)^{-1} \log P_{\theta_0 + u_n b}\bigl( Q_n(\theta_0 + u_n I(\theta_0)^{-1/2} \Omega + u_n b) > c \bigr) = -\inf_{x \in \Omega} |x|^2.
\end{equation}
\end{theorem}

\section{Proofs of Theorems 3.1 and 3.3}

In Section 4.1, we prove the estimate (16) for the approximation of the finite-dimensional distributions of the logarithm of the likelihood ratio. Then in Section 4.2, we obtain estimates of the uniform approximation, which in particular complete the proof of Theorem 3.1. In Section 4.3, the proof of Theorem 3.3 is given.

The proofs of all theorems will be carried out when the true parameter is the parameter \(\theta_0\), and not the sequence of parameters \(\theta_0 + u_n b\). With this course of reasoning, the proofs remain valid for the sequence of parameters \(\theta_0 + u_n b\) if we replace the function \(\varphi_{\theta_0}\) by the function \(\varphi_{\theta_0 + u_n b}\) and \(I(\theta_0)\) by \(I(\theta_0 + u_n b)\).

The statement of the following Lemma 4.1 shows that if in the proofs for the parameters \(\theta_0 + u_n b\) we leave the function \(\varphi_{\theta_0}\) in the arguments instead of \(\varphi_{\theta_0 + u_n b}\), then they retain their force.

\begin{lemma}
\label{lem:4.1}
Let the statistical experiment have finite Fisher information. Then for \(0 < |v| < C |u|\),
\begin{equation*}
\mathbb{E}_{\theta_0 + u} \bigl( g(X_1, u, u+v) - \tfrac12 v^T \varphi_\varepsilon(X_1) \bigr)^2 < C |u|^2 \omega(C_1 |u|).
\end{equation*}
\end{lemma}

From \(\mathbb{E} \varphi^2(X_1) < \infty\), it follows that there exists a function \(q: \mathbb{R}_+ \to \mathbb{R}_+\), \(q(t) \to \infty\) as \(t \to \infty\), such that \(\mathbb{E}[\varphi^2(X_1) q(\varphi(X_1))] < \infty\). Then
\begin{equation*}
\mathbb{E} \varphi_\varepsilon^4(X_1) \le C u_n^{-2} q^{-1}(\varepsilon u_n^{-1}) \mathbb{E}[\varphi_\varepsilon^2(X_1) q(\varphi_\varepsilon(X_1))]
\end{equation*}
and we can assume in the subsequent estimates that \(\omega(u_n)\) is chosen so that \(\omega(u_n) = q^{-1}(\varepsilon u_n^{-1})\).

Note that
\begin{equation*}
g(X_1, u, u+v) = \frac{g(X_1, u+v) + 1}{g(X_1, u) + 1} - 1.
\end{equation*}
From this, using (3), we obtain
\begin{align*}
&\mathbb{E}_{\theta_0 + u} \bigl( g(X_1, u, u+v) - \tfrac12 v^T \varphi_\varepsilon(X_1) \bigr)^2 \\
&= \mathbb{E}_{\theta_0} \Bigl[ \bigl( \tfrac{g(X_1,u+v)+1}{g(X_1,u)+1} - 1 - \tfrac12 v^T \varphi_\varepsilon(X_1) \bigr)^2 (1 + g(X_1,u))^2 \Bigr] \\
&= \mathbb{E}_{\theta_0} \bigl( g(X_1,u+v) + 1 - g(X_1,u) - 1 - \tfrac12 v^T \varphi_\varepsilon(X_1) (1 + g(X_1,u)) \bigr)^2 \\
&\le C \mathbb{E}_{\theta_0} \bigl( g(X_1,u+v) - \tfrac12 (u+v)^T \varphi_\varepsilon(X_1) \bigr)^2 + C \mathbb{E}_{\theta_0} \bigl( g(X_1,u) - \tfrac12 u^T \varphi_\varepsilon(X_1) \bigr)^2 \\
&\quad + C |v|^4 \mathbb{E}_{\theta_0} |\varphi_\varepsilon|^4 + C \mathbb{E}_{\theta_0} \bigl[ (v^T \varphi_\varepsilon(X_1))^2 \bigl( g(X_1,u) - \tfrac12 u^T \varphi_\varepsilon(X_1) \bigr)^2 \bigr] \\
&< C |u|^2 \omega(C_1 |u|).
\end{align*}
From this, in particular, it follows that (16) and (18) follow from (17) and (19) respectively.

\subsection{Proof of (16)}

We begin the proof of (16) when condition \(B2\) is satisfied. Define the random variables \(\xi_{n\varepsilon}(X_i, u) = \xi_n(X_i, u) 1(|\xi_n(X_i, u)| < \varepsilon)\) and \(\overline{\xi}_{n\varepsilon}(X_i, u) = \xi_n(X_i, u) 1(|\xi_n(X_i, u)| > \varepsilon)\), \(1 \le i \le n\).

We have
\begin{equation}
J_n := P\Bigl( \Bigl| \sum_{i=1}^n \xi_n(X_i, u) - \zeta_n(u) \Bigr| > n u_n^2 \omega_1(u_n) \Bigr) \le J_{1n\varepsilon} + J_{2n\varepsilon},
\end{equation}
where
\begin{equation}
J_{1n\varepsilon} = P\Bigl( \Bigl| \sum_{i=1}^n \xi_{n\varepsilon}(X_i, u) - \zeta_n(u) \Bigr| > \tfrac12 n u_n^2 \omega_1(u_n) \Bigr)
\end{equation}
and
\begin{equation}
J_{2n\varepsilon} = P\Bigl( \Bigl| \sum_{i=1}^n \overline{\xi}_{n\varepsilon}(X_i, u) \Bigr| > \tfrac12 n u_n^2 \omega_1(u_n) \Bigr).
\end{equation}

Estimate \(J_{2n\varepsilon}\). We have
\begin{align}
&P\Bigl( \sum_{i=1}^n \overline{\xi}_{n\varepsilon}(X_i, u) > \tfrac12 n u_n^2 \omega_1(u_n) \Bigr) \nonumber \\
&\le \exp\{ -\gamma_n n u_n^2 \omega_1(u_n)/4 \} \bigl( \mathbb{E} \exp\{\gamma_n \overline{\xi}_{n\varepsilon}(X_i, u)/2 \} \bigr)^n \nonumber \\
&= \exp\{ -\gamma_n n u_n^2 \omega_1(u_n)/4 \} \bigl( 1 + \mathbb{E}[(\exp\{\gamma_n \xi_n/2\} - 1) 1(|\xi_n(X_i,u)| > \varepsilon)] \bigr)^n \nonumber \\
&\le \exp\{ -\gamma_n n u_n^2 \omega_1(u_n)/4 \} \nonumber \\
&\quad \times \bigl( 1 + \exp\{-\gamma_n \varepsilon/2\} \mathbb{E}[(\exp\{\gamma_n \xi_n(X_i,u)\} - \exp\{\gamma_n \xi_n(X_i,u)/2\}) 1(|\xi_n(X_i,u)| > \varepsilon)] \bigr)^n \nonumber \\
&\le \exp\{ -\gamma_n n u_n^2 \omega_1(u_n)/4 + C n \exp\{-\gamma_n \varepsilon/2\} \} = \exp\{ -\gamma_n n u_n^2 \omega_1(u_n)/4 (1 + O(1)) \},
\end{align}
where we used condition \(B2\) in the last equality.

Take \(\gamma_{1n} = \gamma_n - 2\). Denote
\begin{equation*}
\tilde{\xi}_{n\varepsilon}(X_1, u) = \log \frac{f(X_1,\theta_0)}{f(X_1,\theta_0 + u)} \, 1\bigl( \bigl| \tfrac{f(X_1,\theta_0)}{f(X_1,\theta_0 + u)} \bigr| > \varepsilon \bigr).
\end{equation*}
Analogously to (46), one estimates
\begin{align}
&P\Bigl( \sum_{i=1}^n \overline{\xi}_{n\varepsilon}(X_i, u) < -\tfrac12 n u_n^2 \omega_1(u_n) \Bigr) = P\Bigl( \sum_{i=1}^n \tilde{\xi}_{n\varepsilon}(X_i, u) > \tfrac12 n u_n^2 \omega_1(u_n) \Bigr) \nonumber \\
&\le \exp\{ -\gamma_{1n} n u_n^2 \omega_1(u_n)/4 \} \bigl( 1 + \mathbb{E}[(\exp\{\gamma_{1n} \tilde{\xi}_n(X_i,u)\} - 1) 1(|\tilde{\xi}_n(X_i,u)| > \varepsilon)] \bigr)^n \nonumber \\
&\le C \exp\{ -\gamma_n n u_n^2 \omega_1(u_n)/4 \} \bigl( 1 + \mathbb{E}_{\theta_0 + u} [(\exp\{\gamma_n \tilde{\xi}_n(X_i,u)\} - 1) 1(|\tilde{\xi}_n(X_i,u)| > \varepsilon)] \bigr)^n \nonumber \\
&\le C \exp\{ -\gamma_n n u_n^2 \omega_1(u_n)/4 \} \times \bigl( 1 + \exp\{-\gamma_n \varepsilon/2\} \mathbb{E}_{\theta_0 + u}[(\exp\{\gamma_n \tilde{\xi}_n\} \nonumber \\
&\quad - \exp\{\gamma_n \tilde{\xi}_n/2\}) 1(|\tilde{\xi}_n| > \varepsilon)] \bigr)^n \nonumber \\
&\le \exp\{ -\gamma_n n u_n^2 \omega_1(u_n)/4 + C n \exp\{-\gamma_n \varepsilon/2\} \} = \exp\{ -\gamma_n n u_n^2 \omega_1(u_n)/4 (1 + o(1)) \}.
\end{align}

Estimate \(J_{1n\varepsilon}\). From Lemma 3.4 in \cite{ermakov2003}, the estimates follow
\begin{equation}
\mathbb{E}[\xi_{n\varepsilon}(X_1,u) - u^T \varphi_\varepsilon(X_1) + \tfrac12 u^T I(\theta_0) u] = O(u_n^2 \omega(u_n))
\end{equation}
and
\begin{equation}
\mathbb{E}[\xi_{n\varepsilon}(X_1,u) - u^T \varphi_\varepsilon(X_1) + \tfrac12 u^T I(\theta_0) u]^2 = O(u_n^2 \omega(u_n)).
\end{equation}
Since
\begin{equation}
|\xi_{n\varepsilon}(X_1,u) - u^T \varphi_\varepsilon(X_1) + \tfrac12 u^T I(\theta_0) u| \le 4\varepsilon
\end{equation}
and (48), (49) hold, applying Prokhorov's inequality \cite{prokhorov1959}, we obtain
\begin{equation}
J_{1n\varepsilon} \le \exp\Bigl\{ -\frac{n u_n^2 \omega_1(u_n)}{8\varepsilon} \log\Bigl\{ \frac{n u_n^2 4\varepsilon}{n u_n^2 \omega(u_n)} \Bigr\} \Bigr\} = \exp\Bigl\{ \frac{n u_n^2 \omega_1(u_n)}{8\varepsilon} \log \frac{\omega(u_n)}{4\varepsilon} \Bigr\}.
\end{equation}
The centered terms are bounded by the constant \(4\varepsilon\), the second moment of the error is \(O(u_n^2 \omega(u_n))\) by (48), (49), and the deviation is of order \(n u_n^2 \omega_1(u_n)\). From this we obtain
\begin{equation}
P\Bigl( \sum_{i=1}^n \xi_{n\varepsilon}(X_i,u) - \zeta_n(u) > n u_n^2 \omega_1(u_n) \Bigr) < C \exp\{ -n u_n^2 / \omega_2(u_n) \},
\end{equation}
if we take \(\omega_1(u_n)\) such that \(\gamma_n \omega_1(u_n) \to \infty\), \(\omega_1(u_n) |\log \omega(u_n)| \to \infty\) as \(n \to \infty\), and set \(\omega_2(u_n)^{-1} = \min\{ \gamma_n \omega_1(u_n), \omega_1(u_n) |\log \omega(u_n)| \}\).

The proof of
\begin{equation}
P\Bigl( \sum_{i=1}^n \xi_{n\varepsilon}(X_i,u) - \zeta_n(u) < - n u_n^2 \omega_1(u_n) \Bigr) < C \exp\{ -n u_n^2 / \omega_2(u_n) \}
\end{equation}
is analogous and is omitted. It suffices to note that
\begin{equation}
\mathrm{Var}_{\theta_0} \log \frac{f(X_i,\theta_0 + u)}{f(X_i,\theta_0)} = \mathrm{Var}_{\theta_0} \log \frac{f(X_i,\theta_0)}{f(X_i,\theta_0 + u)}.
\end{equation}
From (46), (47) and (52), (53), (16) follows.

Let us prove (16) when condition \(B1\) is satisfied. Take a sequence \(\lambda_n\), \(\lambda_n \to \infty\) as \(n \to \infty\) such that \(\varepsilon_n := \lambda_n^{-1} |\log u_n| \to \infty\) and \(\gamma_n / \lambda_n \to \infty\) as \(n \to \infty\). Take \(\omega_1(u_n) \to 0\) as \(n \to \infty\) such that \(\lambda_n \omega_1(u_n) \to \infty\) as \(n \to \infty\). The sequence \(\lambda_n\) is chosen so that \(\lambda_n \to \infty\) and \(\lambda_n = o(|\log u_n|)\). Then \(\varepsilon_n \to \infty\) as \(n \to \infty\). We have
\begin{equation}
J_n \le J_{1n\varepsilon_n} + J_{2n\varepsilon_n}.
\end{equation}
Analogues of estimates (46) and (47) for the upper estimate of \(J_{2n\varepsilon_n}\) differ little. We give only the analogue of estimate (46) in the fragment where some differences arise. We have
\begin{align}
J_{2n\varepsilon_n} &\le \exp\{ -\gamma_n n u_n^2 \omega_1(u_n)/4 \} \bigl( 1 + \exp\{-\gamma_n \varepsilon_n/2\} \nonumber \\
&\quad \times \mathbb{E}[(\exp\{\gamma_n \xi_n(X_i,u)\} - \exp\{\gamma_n \xi_n(X_i,u)/2\}) 1(|\xi_n(X_i,u)| > \varepsilon_n)] \bigr)^n \nonumber \\
&\le \exp\{ -\gamma_n n u_n^2 \omega_1(u_n)/4 + C n \exp\{-\gamma_n \varepsilon_n/2\} \} \nonumber \\
&= \exp\{ -\gamma_n n u_n^2 \omega_1(u_n)/4 (1 + O(1)) \}.
\end{align}
To estimate \(J_{1n\varepsilon_n}\) using Prokhorov's inequality, we need to show that
\begin{equation}
\mathbb{E}[\xi_{n\varepsilon_n}(X_1,u) - u^T \varphi_\varepsilon(X_1) + \tfrac12 u^T I(\theta_0) u] = O(u_n^2 \omega(u_n))
\end{equation}
and
\begin{equation}
\mathbb{E}[\xi_{n\varepsilon_n}(X_1,u) - u^T \varphi_\varepsilon(X_1) + \tfrac12 u^T I(\theta_0) u]^2 = O(u_n^2 \omega(u_n)).
\end{equation}
For the proof of (57) and (58), it suffices to show that
\begin{equation}
\mathbb{E}[\overline{\xi}_{n\varepsilon}(X_1,u) 1(|\xi_n(X_1,u)| < \varepsilon_n)] = O(u_n^2 \omega(u_n))
\end{equation}
and
\begin{equation}
\mathbb{E}[\overline{\xi}_{n\varepsilon}^2 1(|\xi_n(X_1,u)| < \varepsilon_n)] = O(u_n^2 \omega(u_n)),
\end{equation}
which follows from \(C0\) and (3.19)--(3.21) in \cite{ermakov2003}. From (59) and (60), (57) and (58) follow. Using (57) and (58) and applying Prokhorov's inequality, we obtain
\begin{equation}
J_{1n\varepsilon_n} \le \exp\Bigl\{ -\frac{n u_n^2 \omega_1(u_n)}{8 \varepsilon_n} \log\Bigl\{ \frac{n u_n^2 4 \varepsilon_n}{C n u_n^{2+\gamma}} \Bigr\} \Bigr\} \le \exp\{ -C n u_n^2 \omega_1(u_n) \lambda_n (1 + o(1)) \}.
\end{equation}
From (56) and (61), (16) follows.

\subsection{Estimates of the rate of uniform approximation}

Denote
\begin{equation}
Z_{n\theta}(u) = \prod_{j=1}^n \frac{f(X_j, \theta + u)}{f(X_j, \theta)}.
\end{equation}
We will obtain estimates of the moduli of continuity of \(Z_n^{1/2}(u)\) and the logarithm of the likelihood ratio.

We begin the proof of (18). Take a \(\delta_n\)-net on the set \([-C u_n, C u_n]^d\) consisting of points \(u_{ni}\), \(1 \le i \le l_n\), where the value of \(\delta_n\) will be defined later. For the proof of (18), we consider two cases
\begin{equation}
u_n > n^{-1/2} \sqrt{\log n}
\end{equation}
and
\begin{equation}
u_n \le n^{-1/2} \sqrt{\log n}.
\end{equation}

We begin the arguments with the case (63). Take a sequence \(\omega_3(u_n)\) such that \(\omega_3(u_n) \to 0\) and \(\omega_3(u_n)/\omega_2(u_n) \to \infty\) as \(n \to \infty\). Set \(\delta_n = \exp\{ -n u_n^2 / \omega_3(u_n) \}\).

Define the event
\begin{equation}
A_n = \Bigl\{ X_1,\dots,X_n : \max_{1\le i \le l_n} \Bigl| \sum_{i=1}^n \xi_n(X_i, u_{ni}) - \zeta_n(u_{ni}) \Bigr| < n u_n^2 \omega_1(u_n)/2 \Bigr\}.
\end{equation}
We have
\begin{align}
P(\overline{A}_n) &\le \sum_{i=1}^{l_n} P\Bigl( \Bigl| \sum_{i=1}^n \xi_n(X_i, u_{ni}) - \zeta_n(u_{ni}) \Bigr| > n u_n^2 \omega_1(u_n)/3 \Bigr) \nonumber \\
&\le l_n \exp\{ -n u_n^2 / \omega_2(u_n) \} \le \exp\{ -n u_n^2 (1/\omega_2(u_n) - C/\omega_3(u_n)) \} \nonumber \\
&= \exp\{ -n u_n^2 / \omega_2(u_n) (1 + o(1)) \}.
\end{align}
Denote \(w_{in} = n^{1/2} u_{in}\), \(1 \le i \le l_n\).

For the proof of (18), we use the following Lemma 3.1, Ch. 3 in \cite{ibragimov1979}.

\begin{lemma}
\label{lem:4.2}
Let \(Y_1,\dots,Y_n\) be independent identically distributed random variables with \(\mathbb{E} Y_1 = 0\) and \(\mathbb{E}|Y_1|^{2\beta} < \infty\), \(\beta \ge 1\). Then
\begin{equation}
\mathbb{E} \Bigl| \sum_{i=1}^n Y_i \Bigr|^{2\beta} \le C_\beta n^\beta \mathbb{E}|Y_1|^{2\beta}.
\end{equation}
\end{lemma}

Denote \(Y_{ni}(v,u) = \xi_n(X_i, v, u-v) - (u-v)^T \varphi_\varepsilon(X_i)\). From this, using condition \(E\), we obtain
\begin{align}
\mathbb{E} \Bigl| \sum_{i=1}^n (Y_{ni}(v,u) - \mathbb{E} Y_{ni}(v,u)) \Bigr|^{\beta_1} &\le C n^{\beta_1/2} \mathbb{E} |Y_{n1}(v,u) - \mathbb{E} Y_{n1}(v,u)|^{\beta_1} \nonumber \\
&\le C n^{\beta_1/2} |v-u|^{\beta_2}.
\end{align}
Note also that for \(d=1\), condition \(E\) is satisfied with \(\beta_1 = \beta_2 = 2\) by Lemma 3.4 of \cite{ermakov2003}.

Denote
\begin{equation*}
R_n(u) = \sum_{s=1}^n \xi_n(X_s, u) - \zeta_n(u).
\end{equation*}
From this, applying successively (66), Markov's inequality, and then Theorem 19 of Appendix 1 in \cite{ibragimov1979}, we obtain
\begin{align}
&P\Bigl( \sup_i \sup_{|u - u_{in}| < \delta_n} |R_n(u)| > n u_n^2 \omega_3(u_n) \Bigr) \nonumber \\
&\le P\Bigl[ \sup_i |R_n(u_{in})| > n u_n^2 \omega_3(u_n)/2 \Bigr] \nonumber \\
&\quad + P\Bigl( \sup_i \sup_{|u - u_{in}| < \delta_n} |R_n(u) - R_n(u_{in})| > n u_n^2 \omega_3(u_n)/2 \Bigr) \nonumber \\
&\le \exp\{ -n u_n^2 / \omega_2(u_n) \} \nonumber \\
&\quad + C (n u_n^2 \omega_3(u_n))^{-1} \Bigl( \mathbb{E}\Bigl[ \sup_i \sup_{|u-u_{in}|<\delta_n} \bigl| \zeta_n(u) - \zeta_n(u_{in}) + (u_{in}-u) \sum_{s=1}^n \varphi_\varepsilon(X_s) \bigr| \Bigr] \nonumber \\
&\quad + \mathbb{E}\Bigl[ \sup_i \sup_{|w - w_{in}| < \sqrt{n} \delta_n} \Bigl| \sum_{s=1}^n Y_{ns}(n^{-1/2} w_{in}, n^{-1/2} w) - n \mathbb{E}[Y_{n1}(n^{-1/2} w_{in}, n^{-1/2} w)] \Bigr| \Bigr] \nonumber \\
&\quad + n \sup_i \sup_{|u-u_{in}|<\delta_n} \mathbb{E}|Y_{n1}(u_{in},u)| \Bigr) \nonumber \\
&\le \exp\bigl\{ -n u_n^2 / \omega_2(u_n) (1 + o(1)) \bigr\} + C (n u_n^2 \omega_3(u_n))^{-1} \bigl( n u_n \delta_n + (n^{1/2} \delta_n)^{(\beta_2 - d)/\beta_1} + n \delta_n^2 \bigr).
\end{align}
From (66) and (69), we obtain (18) when (63) holds. In the case (64), for the proof of (18), the estimates are analogous and are omitted.

Estimate the modulus of continuity of \(Z_n^{1/2}(u)\) when \(\Theta \subset \mathbb{R}^1\). Take a sequence \(\omega_3(u_n)\) such that \(\omega_3(u_n) \to 0\) and \(\omega_3(u_n)/\omega_2(u_n) \to \infty\) as \(n \to \infty\). Set \(\delta_n = \exp\{ -n u_n^2 / \omega_3(u_n) \}\).

Using the estimate of Lemma 5.5, Ch. 1 in \cite{ibragimov1979}, we obtain
\begin{equation}
\mathbb{E}_{\theta_0} \Bigl( Z_n^{1/2}\bigl( \tfrac{u}{\sqrt{n}} \bigr) - Z_n^{1/2}\bigl( \tfrac{v}{\sqrt{n}} \bigr) \Bigr)^2 \le C |u-v|^2
\end{equation}
for \(|u| < C\), \(|v| < C\) and \(n > n_0\).

Take a sequence \(\omega_4(u_n)\) such that \(\omega_4(u_n) \to 0\), \(\omega_4(u_n)/\omega_3(u_n) \to \infty\) as \(n \to \infty\).

In the case (63), from (70), using Theorem 19 of the Appendix in \cite{ibragimov1979}, we obtain
\begin{align}
&P\Bigl( \sup_i \sup_{|w - w_{in}| < \sqrt{n} \delta_n} \bigl| Z_n^{1/2}(w/\sqrt{n}) - Z_n^{1/2}(w_{in}/\sqrt{n}) \bigr| > \exp\{ -n u_n^2 / \omega_4(u_n) \} \Bigr) \nonumber \\
&\le \exp\{ n u_n^2 / \omega_4(u_n) \} \mathbb{E}\Bigl[ \sup_i \sup_{|w-w_{in}|<\sqrt{n}\delta_n} \bigl| Z_n^{1/2}(w/\sqrt{n}) - Z_n^{1/2}(w_{in}/\sqrt{n}) \bigr| \Bigr] \nonumber \\
&\le C n^{1/4} \exp\{ n u_n^2 / \omega_4(u_n) \} \delta_n^{1/2} \le C \exp\bigl\{ - n u_n^2 / (3 \omega_3(u_n)) \bigr\}.
\end{align}

In the multidimensional case, for the proof of the analogue of inequality (71), instead of inequality (70), the inequality
\begin{equation}
\mathbb{E}_\theta \bigl| Z_{n\theta}^{1/m}\bigl( \tfrac{u}{\sqrt{n}} \bigr) - Z_{n\theta}^{1/m}\bigl( \tfrac{v}{\sqrt{n}} \bigr) \bigr|^m \le C |u-v|^m
\end{equation}
is used (see Lemma 3.2, Chapter 2, \cite{ibragimov1979}). It holds if condition \(D\) is satisfied. Otherwise, the estimates are analogous.

We obtain an estimate of the modulus of continuity of \(Z_n^{1/2}(u)\) in the case (64) and \(\Theta \subset \mathbb{R}^1\). Set \(\delta_n = \alpha_n n^{-1/2}\), where \(\alpha_n = \exp\{ -n u_n^2 / (2 d \omega_2(u_n)) \}\).

Then
\begin{equation}
P(\overline{A}_n) \le \exp\{ -n u_n^2 / \omega_2(u_n) + \log[(n^{1/2} u_n)^d \alpha_n^{-d}] \} = \exp\bigl\{ - n u_n^2 / (2 \omega_2(u_n)) (1 + o(1)) \bigr\}.
\end{equation}
Arguing analogously to (69), and taking \(\omega_4(u_n)\) so that \(\omega_4(u_n)/\omega_2(u_n) \to \infty\) as \(n \to \infty\), we obtain
\begin{align}
&P\Bigl( \sup_i \sup_{|w - w_{in}| < \delta_n \sqrt{n}} \bigl| Z_{n\theta}^{1/m}(w/\sqrt{n}) - Z_{n\theta}^{1/m}(w_{in}/\sqrt{n}) \bigr| > \exp\bigl\{ -(m-d) n u_n^2 / (4 m d \omega_2(u_n)) \bigr\} \Bigr) \nonumber \\
&\le C \alpha_n^{(m-d)/m} \exp\{ (m-d) n u_n^2 / (4 m d \omega_2(u_n)) \} = \exp\{ -(m-d) n u_n^2 / (4 m d \omega_2(u_n)) \}.
\end{align}

\subsection{Proof of Theorem 3.3}

Thus, we have shown that the logarithm of the likelihood ratio, considered as a random process, is well approximated by the random process \(\zeta_n(u)\), \(|u| < C u_n\). An analogous statement holds for the likelihood ratio itself. Thus, the proof of the theorem is essentially reduced to studying the point of attainment of the maximum of the process \(\zeta_n(u)\), \(|u| < C u_n\).

We say that the estimator \(\tilde\theta_n\) is \(u_n\)-consistent if for any \(\theta \in \Theta\), there exists a neighborhood \(U \subset \Theta\) of it such that for any \(\delta > 0\),
\begin{equation*}
\lim_{n\to\infty} \sup_{\theta \in U} P_\theta(|\tilde\theta - \theta| > \delta u_n) = 0.
\end{equation*}
If the estimator is \(u_n\)-consistent and the Fisher information is finite, then for it the lower bound of asymptotic Bahadur efficiency in the zone of moderate deviation probabilities holds (see Theorem 2.3, \cite{ermakov2012}). This lower bound coincides with the lower bound of the moderate deviation probabilities of the maximum likelihood estimator in Theorem 3.3. If the conditions of Theorem 3.2, Ch. 2 in \cite{ibragimov1979} are satisfied, then the maximum likelihood estimator is asymptotically normal and consequently \(u_n\)-consistent. If the conditions of Theorem 3.3 are satisfied, then the conditions of Theorem 3.2, Ch. 2 in \cite{ibragimov1979} are also satisfied. Thus, for the proof of (24), it remains only to prove the upper bound.

Note that by Theorem 5.1, Ch. 1 in \cite{ibragimov1979}, for any sequence \(r_n\), \(r_n \to \infty\), \(r_n u_n \to 0\) as \(n \to \infty\),
\begin{equation*}
\lim_{n\to\infty} (n u_n^2)^{-1} \log P(|\hat\theta_n - \theta_0| > r_n u_n) = -\infty.
\end{equation*}

Define the event
\begin{align*}
D_n &= \{ X_1,\dots,X_n : |Z_n(u) - Z_n(u_{in})| \le C \exp\{-2 n u_n^2\}, \\
&\quad \exp\{\zeta_n(u_{in}) - n u_n^2 \omega_1(u_n)\} < Z_n(u_{in}) < \exp\{\zeta_n(u_{in}) + n u_n^2 \omega_1(u_n)\}, \\
&\quad |u - u_{in}| < \delta_n,\ 1 \le i \le l_n \}.
\end{align*}
From (66), (73) and (71), (74), it follows that
\begin{equation}
P(\overline{D}_n) \le C \exp\{ -C_1 n u_n^2 \},
\end{equation}
where the constant \(C_1\) can be taken arbitrarily large for \(C = C(C_1)\).

For the proof of the upper bound, we show that if for \(0 < r < 1\), for the event
\begin{equation*}
\Upsilon_n(r) = \{ X_1,\dots,X_n : |\psi_n| < n^{1/2} r u_n \},
\end{equation*}
\begin{equation}
P\bigl( |I^{1/2}(\theta_0)(\hat\theta_n - \theta_0)| > u_n,\ D_n \cap \Upsilon_n(r) \bigr) = 0.
\end{equation}
By Theorem 3.2, for the complementary events \(\overline{\Upsilon}_n(r)\) to the event \(\Upsilon_n(r)\),
\begin{equation}
\lim_{n\to\infty} (n u_n^2)^{-1} \log P(\overline{\Upsilon}_n(r)) = -r^2/2.
\end{equation}
Consequently, from (75)--(77) and the fact that
\begin{equation*}
P\bigl( |I^{1/2}(\theta_0)(\hat\theta_n - \theta_0)| > u_n,\ D_n \bigr) \ge P\bigl( I^{1/2}(\theta_0)(\hat\theta_n - \theta_0) \in u_n \Omega,\ D_n \bigr)
\end{equation*}
if \(\inf\{ |x| : x \in \Omega \} > 1\), the validity of the upper bound in (24) follows.

Note that if the event \(D_n\) holds, then the events \(|\psi_n| < n^{1/2} u_n (1-\varepsilon)/2\), \(0 < \varepsilon < 1\), and \(|I^{1/2}(\theta_0)(\hat\theta_n - \theta_0)| > u_n\) cannot hold simultaneously. They would simultaneously mean that \(Z_n(\hat\theta_n) < 1\) and \(Z_n(0) = 1\). Thus, we arrive at a contradiction.

Therefore, in what follows, we will assume that the event
\begin{equation}
G_n = \{ X_1,\dots,X_n : |\psi_n| > n^{1/2} u_n / 2 \}
\end{equation}
holds.

Take \(\kappa\), \(0 < \kappa < r\), and define the events
\begin{equation}
W_n(\kappa) = D_n \cap \{ X_1,\dots,X_n : |\psi_n| < (r - \kappa) n^{1/2} u_n \}.
\end{equation}
Note that
\begin{equation}
\max_{u^T I(\theta_0) u < (r-\kappa) u_n^2} \zeta_n(u) = -\tfrac12 \min_{u^T I(\theta_0) u < (r-\kappa) u_n^2} |\psi_n - \sqrt{n} u^T I^{1/2}(\theta_0)|^2 + \tfrac12 \psi_n^T \psi_n = \tfrac12 \psi_n^T \psi_n.
\end{equation}

Assume that case (63) holds. Then on the set of events \(W_n(\kappa)\), for the point \(u_{in}\), \(|u_{in} - \hat\theta_n| < \delta_n\),
\begin{align}
&\bigl| \log Z_n^{1/2}(\hat\theta_n) - \log Z_n^{1/2}(u_{in}) \bigr| \nonumber \\
&\le \bigl| \log\bigl(1 + Z_n^{-1/2}(u_{in}) (Z_n^{1/2}(\hat\theta_n) - Z_n^{1/2}(u_{in}))\bigr) \bigr| \nonumber \\
&\le C Z_n^{-1/2}(u_{in}) (Z_n^{1/2}(\hat\theta_n) - Z_n^{1/2}(u_{in})) \nonumber \\
&\le C \exp\{ -c n u_n^2 / \omega_3(u_n) \},
\end{align}
where the second inequality follows from the fact that \(Z_n^{-1/2}(u_{in})(Z_n^{1/2}(\hat\theta_n) - Z_n^{1/2}(u_{in}))\) is positive, since \(\hat\theta_n\) is the maximum likelihood estimator, and the third inequality from the fact that the event \(D_n\) holds.

Consider case (64). Then, estimating analogously to (81) and applying analogous arguments, we obtain
\begin{align}
&\bigl| \log Z_n^{1/2}(\hat\theta_n) - \log Z_n^{1/2}(u_{in}) \bigr| \nonumber \\
&\le C Z_n^{-1/2}(u_{in}) (Z_n^{1/2}(\hat\theta_n) - Z_n^{1/2}(u_{in})) \nonumber \\
&\le C \exp\{ -c n u_n^2 / \omega_3(u_n) \cdot (2-r)/2 \}.
\end{align}
From (81) and (82), it follows that on the set of events \(D_n\), for the point \(u_{in}\), \(|u_{in} - \hat\theta_n| < \delta_n\),
\begin{equation}
\bigl| \log Z_n^{1/2}(\hat\theta_n) - \log Z_n^{1/2}(u_{in}) \bigr| \le o(n u_n^2).
\end{equation}
Therefore, on the set of events \(D_n \cap W_n(\kappa)\),
\begin{equation}
\max_{u^T I(\theta_0) u < (r-\kappa) u_n^2} \log Z_n(u) = \tfrac12 \psi_n^T \psi_n (1 + o(1))
\end{equation}
and the maximum is attained in the region \(\{ u : u^T I(\theta_0) u < r^2 u_n^2 \}\), which completes the proof of (24).

Let \(e \in \mathbb{R}^d\) and \(|e| = 1\). Denote \(B_\delta\) the ball of radius \(\delta\), \(0 < \delta < 0.05\), centered at zero.

For the proof of (27)--(29), it suffices to show that
\begin{equation}
\lim_{n\to\infty} (n u_n^2)^{-1} \log P_\theta\bigl( \psi_n \in u_n n^{1/2} (e + B_\delta),\ I^{1/2}(\theta_0)(\hat\theta_n - \theta_0) \notin u_n (e + B_{4\sqrt{\delta}}) \bigr) = -\infty.
\end{equation}
For simplification of the symbolism, we carry out the arguments for the proof of (85) when the matrix \(I(\theta_0)\) is the identity. It suffices to prove (85) when the event \(D_n\) holds. Then, if \(\psi_n \in n^{1/2} u_n (e + B_\delta)\), the smallest possible value of \(\max\{\zeta_n(u),\ u \in \mathbb{R}^d\}\) is equal to \((1-\delta)^2 n u_n^2 / 2\). At the same time, the largest possible value of \(\max\{\zeta_n(u),\ u \notin u_n (e + B_{4\sqrt{\delta}})\}\) does not exceed \(n u_n^2 ((1+\delta)^2 - 8\delta)/2\). Consequently, (85) holds.

\section{Proof of Theorem 3.5}

Since under the conditions of Theorem 3.5 the Bayesian estimator is asymptotically normal (see Theorem 3.1, Ch. 3, \cite{ibragimov1979}), it is \(u_n\)-consistent. Consequently, the lower bound in the analogue of (24) follows from the lower bound for asymptotic Bahadur efficiency in the zone of moderate deviation probabilities (see Theorem 2.3, \cite{ermakov2012}).

The proof of the upper bound of the analogue of (24) will be based on the same approach as the proof of Theorem 3.3. We define an event \(D_n\) such that \((n u_n^2)^{-1} \log P_{\theta_0}(\overline{D}_n) \to -\infty\) as \(n \to \infty\), and show that for any \(\delta > 0\),
\begin{equation}
P_{\theta_0}\bigl( \|\tilde\theta_n - \theta_0\| > u_n,\ \|n^{-1/2} I^{-1/2}(\theta_0) \psi_n\| < (1-\delta) u_n,\ D_n \bigr) = 0.
\end{equation}
Then the upper bound will follow from Theorem 3.2.

For \(r > 0\), denote
\begin{equation}
I(r) = n^{d/2} \int_{|u| > r u_n} Z_n(u) \pi(u) \, du
\end{equation}
and
\begin{equation}
J(r) = n^{(a+d)/2} \int_{|u| > r u_n} |u|^a Z_n(u) \pi(u) \, du, \quad a > 0.
\end{equation}

From Lemma 5.2, Ch. 1, \cite{ibragimov1979} and Theorem 5.5, Ch. 1, \cite{ibragimov1979}, it follows that for any \(a\), one can choose constants \(C_1\) and \(C_2\) such that for any sequence \(r_n\) with \(r_n u_n \to 0\), \(r_n \to \infty\) as \(n \to \infty\), for \(n > n_0\),
\begin{equation}
P\bigl( I(r_n) > \exp\{ -C_1 n r_n^2 u_n^2 \} \bigr) \le \exp\{ -C_2 n r_n^2 u_n^2 \},
\end{equation}
\begin{equation}
P\bigl( J(r_n) > \exp\{ -n C_1 r_n^2 u_n^2 \} \bigr) \le \exp\{ -C_2 n r_n^2 u_n^2 \},
\end{equation}
and moreover there exists \(C_3\) such that
\begin{equation}
P\bigl( \|\tilde\theta_n - \theta_0\| > r_n u_n / 2 \bigr) \le \exp\{ -C_3 n r_n^2 u_n^2 \}.
\end{equation}

In Theorem 5.5, Ch. 1, \cite{ibragimov1979}, inequality (91) is proved for the loss function \(l_1(|\theta|)\). However, the proof practically does not differ for the loss function \(l_1(\|\theta\|)\) considered by us.

In the subsequent comparisons of posterior risks, the integration is restricted to the region \(|\theta - \theta_0| < r_n u_n\). By (89), (90), the contribution of the complement has a smaller order than the positive differences of risks obtained below in (106) and (121).

In the subsequent estimates, we will assume that \(r_n \to \infty\) sufficiently slowly as \(n \to \infty\) so that this does not prevent them from being valid. Such restrictions on the sequence \(r_n\) are finite, so such sequences \(r_n\) exist.

We carry out the further arguments for the case
\begin{equation}
\frac{u_n}{\sqrt{n / \log n}} \to \infty \quad \text{as}\quad n \to \infty,
\end{equation}
and then indicate the differences in the case
\begin{equation}
\frac{u_n}{\sqrt{n / \log n}} < c,
\end{equation}
where \(c\) is an arbitrary positive constant.

Instead of estimates of the modulus of continuity and application of Theorem 19 of Appendix 1 in \cite{ibragimov1979}, our estimates will be based on the following auxiliary inequality
\begin{align}
\mathbb{E}|Z_n(u) - Z_n(v)| &= \mathbb{E} \bigl| (Z_n^{1/2}(u) - Z_n^{1/2}(v)) (Z_n^{1/2}(u) + Z_n^{1/2}(v)) \bigr| \nonumber \\
&\le \bigl( \mathbb{E}(Z_n^{1/2}(u) - Z_n^{1/2}(v))^2 \bigr)^{1/2} \bigl( \mathbb{E}(Z_n^{1/2}(u) + Z_n^{1/2}(v))^2 \bigr)^{1/2} \nonumber \\
&\le 2 \bigl( \mathbb{E}(Z_n^{1/2}(u) - Z_n^{1/2}(v))^2 \bigr)^{1/2} \le C n^{1/2} |u-v|.
\end{align}
Here we used Lemma 5.5, Ch. 1, \cite{ibragimov1979} in the last inequality.

Take a sequence \(r_n > 0\), \(r_n \to \infty\), \(r_n u_n \to 0\) as \(n \to \infty\). Define on the set \(K_n = (-r_n u_n, r_n u_n)^d\) a \(\delta_n\)-net consisting of points \(u_{ni}\), \(1 \le i \le l_n\). The definition of the sequence \(\delta_n > 0\) will be different for cases (92) and (93). Partition the set \(K_n\) into measurable non-intersecting subsets \(\Delta_{ni}\), \(u_{ni} \in \Delta_{ni}\), \(1 \le i \le k_n\), of diameter \(\sup_{u,v \in \Delta_{ni}} |u-v| < 3\delta_n\).

Define the event \(A_n\) by equality (65). Denote
\begin{equation}
L_n = \int_{K_n} Z_n(u) \, du - \sum_{i=1}^{k_n} \int_{\Delta_{ni}} Z_n(u_{ni}) \, du.
\end{equation}
Using (94), we obtain
\begin{equation}
J_n = \mathbb{E}|L_n| \le \sum_{i=1}^{k_n} \int_{\Delta_{ni}} \mathbb{E}|Z_n(u) - Z_n(u_{ni})| \, du \le C n^{1/2} \delta_n (r_n u_n)^d.
\end{equation}
Note that in (96), the integration over the cube \(K_n\) can be replaced by integration over any union of the sets \(\Delta_{ni}\).

Denote \(C_0 = \sup\{ |I^{1/2}(\theta_0) u| : \|u\| = 1,\ u \in \mathbb{R}^d \}\).

Thus, we begin with case (92). Compare the posterior risks when \(t = 0\) and when \(\|t\| > u_n\).

Let \(\|t\| > u_n\). For \(\|u\| < u_n/2\), we have \(\|t - u\| \ge \|t\| - \|u\| > \|u\|\). Consequently, for \(0 < r_1 < 1/2\),
\begin{equation}
\int_{\|u\| < r_1 u_n} l(t-u) Z_n(u) \, du \ge \int_{\|u\| < r_1 u_n} l(u) Z_n(u) \, du.
\end{equation}

For \(0 < \delta_0 < 0.01\), consider the case when the event
\begin{equation}
\Upsilon_n = \{ X_1,\dots,X_n : |\psi_n| < \delta_0 n^{1/2} u_n \}
\end{equation}
holds.

Take \(r_1\), \(r_2\) such that \(0 < C_0 r_1 < r_2 < 1/2\). Then
\begin{equation*}
U_n = \{ u : C_0 r_1 u_n < |I^{1/2}(\theta_0) u| < r_2 u_n \} \subset V_n = \{ u : r_1 u_n \le \|u\| \le r_2 u_n \}.
\end{equation*}
Consequently, by (38),
\begin{equation}
\int_{V_n} (l(t-u) - l(u)) Z_n(u) \, du \ge C l_1(u_n) \int_{U_n} Z_n(u) \, du.
\end{equation}

Take a sequence \(\omega_3(u_n)\) such that \(\omega_3(u_n) \to 0\) and \(\omega_3(u_n)/\omega_2(u_n) \to \infty\) as \(n \to \infty\). Set \(\delta_n = \exp\{ -n u_n^2 / \omega_3(u_n) \}\). Denote \(M_n\) the set of points \(u_{ni}\), \(1 \le i \le l_n\), such that \(u_{ni} \in U_n\). Denote \(W_{1n} = \cup_{u_{in} \in M_n} \Delta_{ni}\).

Set
\begin{equation}
L_{1n} = L_n(r_1, r_2) = \int_{W_{1n}} Z_n(u) \, du - \sum_{u_{ni} \in M_n} \int_{\Delta_{ni}} Z_n(u_{ni}) \, du.
\end{equation}
Then, estimating analogously to (96), we obtain
\begin{equation}
\mathbb{E}|L_{1n}| \le C n^{1/2} \delta_n u_n^d.
\end{equation}
Define the event \(D_{1n} = D_n(r_1, r_2) = \{ X_1,\dots,X_n : |L_{1n}| < \exp\{ -n u_n^2 / (3 \omega_3(u_n)) \} \}\). Then, using (101), we obtain
\begin{equation}
P(\overline{D}_{1n}) < \exp\{ n u_n^2 / (3 \omega_3(u_n)) \} \mathbb{E}|L_{1n}| \le C \exp\{ -2 n u_n^2 / (3 \omega_3(u_n)) \}.
\end{equation}

Further arguments will be carried out assuming that the event \(A_n \cap D_{1n}\) holds.

Since the event \(A_n\) holds,
\begin{equation}
Z_n(u_{ni}) \exp\{ -\zeta(u_{ni}) \} \ge \exp\{ -n u_n^2 \omega_1(u_n) \}.
\end{equation}
Consequently,
\begin{equation}
Z_n(u_{ni}) \ge \exp\{ -2 n |u_{ni}| \delta_0 u_n - n u_{ni}^T I(\theta_0) u_{ni}/2 - n u_n^2 \omega_1(u_n) \}.
\end{equation}
From this we obtain
\begin{equation}
\sum_{u_{ni} \in M_n} \int_{\Delta_{ni}} Z_n(u_{ni}) \, du \ge n^{-d/2} \exp\{ -n (r_2^2 + 4 \delta_0) u_n^2 / 2 - n u_n^2 \omega_1(u_n) (1 + o(1)) \}.
\end{equation}
From (99) and (105), we obtain
\begin{equation}
\int_{V_{1n}} (l(t-u) - l(u)) Z_n(u) \, du \ge l_1(u_n) n^{-d/2} \exp\{ -n (r_2^2 + 4 \delta_0) u_n^2 / 2 - n u_n^2 \omega_1(u_n) (1 + o(1)) \}.
\end{equation}

We have
\begin{equation}
\int_{r_2 u_n \le \|u\| < u_n/2} (l(t-u) - l(u)) Z_n(u) \, du > 0
\end{equation}
and
\begin{equation}
\int_{V_n \setminus U_n} (l(t-u) - l(u)) Z_n(u) \, du > 0.
\end{equation}

Estimate from above
\begin{equation}
\int_{u_n/2 \le \|u\| < r_n u_n} |l(t-u) - l(u)| Z_n(u) \, du.
\end{equation}
Denote \(M_{2n}\) a \(\delta_n\)-net consisting of points \(u_{ni}\), \(1 \le i \le l_{1n}\), such that \(0.5 u_n \le \|u_{ni}\| \le r_n u_n\). Define the event \(D_{2n} = D_n(0.5, r_n)\). Analogously to (102), we have
\begin{equation}
P(\overline{D}_{2n}) \le C \exp\{ -2 n u_n^2 / (3 \omega_3(u_n)) (1 + o(1)) \}.
\end{equation}
If the event \(A_n\) holds, then for the points \(u_{ni}\), \(1 \le i \le l_{1n}\),
\begin{equation}
Z_n(u_{ni}) \exp\{ -\zeta_n(u_{ni}) \} \le \exp\{ n |u_n|^2 \omega_1(|u_n|) \}.
\end{equation}
Consequently,
\begin{equation}
Z_n(u_{ni}) \le \exp\{ 2 n |u_{ni}| \delta_0 u_n - n u_{ni}^T I(\theta_0) u_{ni}/2 + n u_n^2 \omega_1(u_n) \}.
\end{equation}
Since \(\|u\| \le |I^{1/2} u|\) for \(u \in \mathbb{R}^d\), from this we obtain
\begin{equation}
\sum_{u_{ni} \in M_{2n}} \int_{\Delta_{ni}} Z_n(u_{ni}) \, du \le \exp\{ -n (1/4 - 4 \delta_0) u_n^2 / 2 + n u_n^2 \omega_1(u_n) (1 + o(1)) \}.
\end{equation}
Consequently, if the event \(A_n \cap D_{2n}\) holds, then
\begin{align}
&\int_{u_n/2 \le \|u\| < r_n u_n} |l(t-u) - l(u)| Z_n(u) \, du \le C l_1(r_n u_n) \int_{u_n/2 \le \|u\| < r_n u_n} Z_n(u) \, du \nonumber \\
&\le C l_1(r_n u_n) n^{-d/2} \exp\{ -n (1/8 - 2 \delta_0) u_n^2 + n u_n^2 \omega_1(u_n) (1 + o(1)) \}.
\end{align}
Since the right-hand side of (106) is greater than the right-hand side of (114), from (97), (106), (107) and (114), it follows that the posterior risk at \(t=0\) is less than at \(\|t\| > u_n\). This means (86), provided that the event \(\Upsilon_n \cap A_n \cap D_{1n} \cap D_{2n}\) holds.

Take \(\delta > 0\), \(\delta < \delta_0/10\). Consider the case \(\delta_0 n^{1/2} u_n < \|I^{-1/2}(\theta_0) \psi_n\| < (1 - 5\delta) n^{1/2} u_n\).

Denote \(B(v,C) = \{ x : \|x - v\| < C,\ x \in \mathbb{R}^d \}\) the ball centered at the point \(v \in \mathbb{R}^d\) of radius \(C\).

Let the points \(x_{n1},\dots,x_{nk}\) form a \(\delta u_n\)-net in the set \(B(0, (1-5\delta) u_n) \setminus B(0, \delta_0 u_n)\), where the distance is considered with respect to the norm \(\|\cdot\|\). Denote \(J_{ni}\) the set of all such indices \(j\), \(1 \le j \le l_n\), that \(\Delta_{nj} \subset B(x_{ni}, 2\delta u_n)\). Denote \(\overline{J}_{ni} = \{ j : j \notin J_{ni},\ 1 \le j \le l_n \}\), \(1 \le i \le k\).

Set \(U_{ni} = \cup_{j \in J_{ni}} \Delta_{nj}\) and \(V_{ni} = K_n \setminus U_{ni}\), \(1 \le i \le k\).

For \(1 \le i \le k\), denote
\begin{equation}
L_{ni} = \int_{U_{ni}} Z_n(u) \, du - \sum_{j \in J_{ni}} \int_{\Delta_{nj}} Z_n(u_{nj}) \, du,
\end{equation}
\begin{equation}
\overline{L}_{ni} = \int_{V_{ni}} Z_n(u) \, du - \sum_{j \in \overline{J}_{ni}} \int_{\Delta_{nj}} Z_n(u_{nj}) \, du.
\end{equation}
Define the events
\begin{equation}
D_{ni} = \{ X_1,\dots,X_n : |L_{ni}| < \exp\{ -n u_n^2 / (3 \omega_3(u_n)) \} \},
\end{equation}
\begin{equation}
G_{ni} = \{ X_1,\dots,X_n : |\overline{L}_{ni}| < \exp\{ -n u_n^2 / (3 \omega_3(u_n)) \} \}.
\end{equation}
Estimating analogously to (96) and (102), we obtain
\begin{equation}
P(\overline{D}_{ni}) < C \exp\{ -2 n u_n^2 / (3 \omega_3(u_n)) \}, \quad P(\overline{G}_{ni}) < C \exp\{ -2 n u_n^2 / (3 \omega_3(u_n)) \}
\end{equation}
for \(1 \le i \le k\).

Denote \(D_{3n} = \cap_{i=1}^k (D_{ni} \cap G_{ni})\).

Let \(\| n^{-1/2} I^{-1/2}(\theta_0) \psi_n - x_{ni} \| < \delta u_n\). We show that then, if the event \(A_n \cap D_{3n}\) holds,
\begin{equation}
J_n = \inf_{\|t\| > u_n} \int l(t-u) Z_n(u) \, du - \int l(x_{ni} - u) Z_n(u) \, du > 0,
\end{equation}
which will complete the proof.

For \(u \in U_{ni}\), we have \(\|u - x_{ni}\| < 2 \delta u_n\),
\begin{equation*}
\|t - u\| \ge \|t\| - \|x_{ni}\| - \|u - x_{ni}\| > u_n - (1-5\delta) u_n - 2\delta u_n = 3\delta u_n,
\end{equation*}
which means \(l(x_{ni} - u) \le l_1(2\delta u_n)\) and \(l(t-u) \ge l_1(3\delta u_n)\).

Consequently,
\begin{align}
J_n &\ge l_1(3\delta u_n) \int_{U_{ni}} Z_n(u) \, du + \inf_{\|t\| > u_n} \int_{V_{ni}} l(t-u) Z_n(u) \, du \nonumber \\
&\quad - l_1(2\delta u_n) \int_{U_{ni}} Z_n(u) \, du - \int_{V_{ni}} l(x_{ni}-u) Z_n(u) \, du \nonumber \\
&\ge (l_1(3\delta u_n) - l_1(2\delta u_n)) \int_{U_{ni}} Z_n(u) \, du - \int_{V_{ni}} l(x_{ni}-u) Z_n(u) \, du \nonumber \\
&\ge (l_1(3\delta u_n) - l_1(2\delta u_n)) \int_{U_{ni}} Z_n(u) \, du - l_1(r_n u_n) \int_{V_{ni}} Z_n(u) \, du.
\end{align}
Using (119) and the fact that the event \(A_n\) holds, we obtain
\begin{equation}
\int_{U_{ni}} Z_n(u) \, du = n^{-d/2} \exp\{ |\psi_n|^2 / 2 + o(n u_n^2) \}
\end{equation}
and, since \(\|u - n^{-1/2} I^{-1/2}(\theta_0) \psi_n\| \ge \delta u_n / 2\) for \(u \in V_{ni}\),
\begin{equation}
\int_{V_{ni}} Z_n(u) \, du \le n^{-d/2} \exp\{ |\psi_n|^2 / 2 - c \delta^2 n u_n^2 (1 + o(1))/2 \}.
\end{equation}
From (121)--(123), (120) follows.

Thus, (86) holds if we take \(D_n = A_n \cap D_{1n} \cap D_{2n} \cap D_{3n}\).

Consider case (93). Set \(\delta_n \asymp n^{-1/2} \alpha_n\), where \(\alpha_n = \exp\{ -n u_n^2 / (2 d \omega_2(u_n)) \}\). With this choice of \(\delta_n\), if we prove analogues of inequalities (102), (110) and (119), then all estimates (106), (114) and (121)--(123) will be valid.

Define the events
\begin{equation}
\overline{D}_{21n} = D_{2n}(r_1,r_2) = \{ X_1,\dots,X_n : |L_{1n}| > n^{-d/2} \alpha_n^{1/2} r_n \},
\end{equation}
and set \(D_{2n} = D_n(0.5, r_n)\).

For \(1 \le i \le k\), define the events
\begin{equation}
\overline{D}_{ni} = \{ X_1,\dots,X_n : |L_{ni}| > n^{-d/2} \alpha_n^{1/2} \},
\end{equation}
\begin{equation}
G_{ni} = \{ X_1,\dots,X_n : |\overline{L}_{ni}| > n^{-d/2} \alpha_n^{1/2} \}.
\end{equation}
Arguing analogously to (96), we obtain
\begin{equation}
\mathbb{E}|L_{1n}| < C n^{1/2} \alpha_n n^{-1/2} (r_n u_n)^d
\end{equation}
and for \(1 \le i \le k\),
\begin{equation}
\mathbb{E}|L_{ni}| < C \alpha_n u_n^d, \qquad \mathbb{E}|\overline{L}_{ni}| < C \alpha_n r_n^d u_n^d.
\end{equation}
From this, estimating analogously to (102), we obtain
\begin{equation}
P(\overline{D}_{21n}) < C \alpha_n^{1/2} n^{d/2} (r_n u_n)^d, \qquad P(D_{22n}) < C \alpha_n^{1/2} n^{d/2} (r_n u_n)^d
\end{equation}
and for \(1 \le i \le k\),
\begin{equation}
P(\overline{D}_{2ni}) < C \alpha_n^{1/2} n^{d/2} u_n^d, \qquad P(G_{2ni}) < C \alpha_n^{1/2} r_n^d n^{d/2} u_n^d.
\end{equation}
Since these asymptotics are substantially smaller in order than the asymptotics
\begin{equation}
P_{\theta_0}\Bigl( \bigl| I^{-1/2}(\theta_0) \sum_{i=1}^n \varphi_\varepsilon(X_i) \bigr| > n u_n \Bigr),
\end{equation}
from this the upper bound in the analogue of (24) for Bayesian estimators follows.

Let us prove (31)--(33) for \(b=0\). In (98), we can always take \(\delta_0 < \delta\). In (121)--(123), we essentially showed that the event \(\| n^{-1/2} I^{-1/2}(\theta_0) \psi_n - x_{ni} \| < \delta u_n\) and \(\| \tilde\theta_n - \theta_0 - x_{ni} \| > 5 \delta u_n\) has probability not exceeding \(\exp\{ -C n u_n^2 \}\), where \(C\) is arbitrary. We can carry out the same estimate on the set of events \(\| \tilde\theta_n - \theta_0 \| < R u_n\) and \(|\psi_n| < n^{1/2} R u_n\). Letting \(R\) tend to infinity and estimating \(P_{\theta_0}(|\psi_n| \ge n^{1/2} R u_n)\) using (91) and Theorem 3.2, we obtain that
\begin{equation}
\lim_{n\to\infty} (n u_n^2 / 2)^{-1} \log P_{\theta_0}\bigl( |I^{1/2}(\theta_0)(\tilde\theta_n - \theta_0) - n^{-1/2} \psi_n| > \delta u_n \bigr) = -\infty,
\end{equation}
\begin{equation}
\lim_{n\to\infty} (n u_n^2 / 2)^{-1} \log P_{\theta_0}\Bigl( \Bigl| \sum_{i=1}^n \xi_n(X_i, \tilde\theta_n - \theta_0) - \psi_n^T \psi_n / 2 \Bigr| > \delta n u_n^2 \Bigr) = -\infty
\end{equation}
and
\begin{equation}
\lim_{n\to\infty} (n u_n^2 / 2)^{-1} \log P_{\theta_0}\Bigl( \Bigl| \sum_{i=1}^n \xi_n(X_i, \tilde\theta_n - \theta_0) - n (\tilde\theta_n - \theta_0)^T I(\theta_0) (\tilde\theta_n - \theta_0)/2 \Bigr| > \delta n u_n^2 \Bigr) = -\infty.
\end{equation}
Equalities (133), (134) follow from the obtained exponential equivalence and the uniform approximation of the logarithm of the likelihood ratio provided by condition \(E\) by Theorem 3.1 or by condition \(D\) via the estimates of Sections 4.2 and 4.3.

\subsection{Proof of Theorem 3.6}

The arguments are largely based on the estimates of the proof of Theorem 3.5. We will adhere to the notation of the proof of Theorem 3.5. By a linear change of coordinates, we assume \(\theta_0 = 0\), \(I(\theta_0) = E_d\). The symbol \(B\) is used to denote Euclidean balls.

We begin with the proof of the lower bound in (42).

Denote \(r = \inf \{ |\tau| : \tau \in \Omega \}\).

For any \(\varepsilon > 0\), there exists a point \(e \in \Omega\) and \(\delta > 0\) such that \(|e| - r < \varepsilon\) and \(B(e, 6\delta) \subset \Omega\).

Define the events
\begin{equation}
\Psi_{ne} = \{ X_1,\dots,X_n : \psi_n \in n^{1/2} u_n B(e,\delta) \}.
\end{equation}
Then by Theorem 3.2, we obtain
\begin{equation}
\lim_{n\to\infty} (n u_n^2 / 2)^{-1} \log P(\Psi_{ne}) = -\inf_{x \in B(e,\delta)} |x|^2.
\end{equation}
Define the event
\begin{equation}
G_n = \Bigl\{ X_1,\dots,X_n : \log \Bigl( n^{d/2} \int_{\overline{B}(u_n e, 5\delta u_n)} Z_n(u) \pi(u) \, du \Bigr) > \psi_n^T \psi_n / 2 - \delta^2 n u_n^2 \Bigr\}.
\end{equation}
Then, estimating analogously to (119), (122) and (123), we obtain
\begin{equation}
\lim_{n\to\infty} (n u_n^2 / 2)^{-1} \log P(\Psi_{ne} \cap G_n) = -\infty.
\end{equation}
For \(0 < \varepsilon < 1\), define the event
\begin{equation}
G_{1n\varepsilon} = \Bigl\{ X_1,\dots,X_n : \bigl| (\psi_n^T \psi_n / 2)^{-1} \log \bigl( n^{d/2} \int_{B(u_n e, \delta u_n)} Z_n(u) \pi(u) \, du \bigr) - 1 \bigr| < \varepsilon \Bigr\}.
\end{equation}
Then, estimating analogously to (119), (122) and (123), we obtain that for any \(\varepsilon > 0\),
\begin{equation}
\lim_{n\to\infty} (n u_n^2 / 2)^{-1} \log P(\Psi_{ne} \cap \overline{G}_{1n\varepsilon}) = -\infty.
\end{equation}
From (136)--(140), the lower bound in Theorem 3.6 follows.

In the works \cite{arcones2002} and \cite{ibragimov1979}, moderate deviation probabilities are analyzed in a different normalization, namely in our notation for \(n^{1/2} u_n\), considering the random processes \(Z_n(n^{-1/2} u)\). This allows one to avoid the additional normalization \(n^{d/2}\) in (137), (139) and other places of the work. However, we follow \cite{ermakov2003,ermakov2012,ermakov2025}.

Applying the methods of the proof of Theorem 3.5, namely inequalities (102), (105), (110), (113) when the event \(\Upsilon_n\) holds, and equalities (122), (123) when the event \(\overline{\Upsilon}_n\) holds, we obtain that for any \(C > 0\) for the set
\begin{equation}
G_{2n} = \{ X_1,\dots,X_n : |\psi_n| < (r - \delta) n^{1/2} u_n \},
\end{equation}
\begin{equation}
\lim_{n\to\infty} (n u_n^2 / 2)^{-1} \log P\Bigl( \frac{\int_{|u| < r u_n} Z_n(u) \pi(u) \, du}{\int_{|u| > r u_n} Z_n(u) \pi(u) \, du} < C,\ G_{2n} \Bigr) = -\infty,
\end{equation}
which together with Theorem 3.2 gives the upper bound in Theorem 3.6.

\section*{Acknowledgments}

The author expresses gratitude to Grok (xAI) and ChatGPT (AI) for assistance in checking intermediate calculations and editing the article.

\end{document}